\documentclass[proceedings]{aofa}
\pdfoutput = 1
\newif\iffinal
\finaltrue
\iffinal\else
\usepackage[eso-foot]{svninfo}
\svnInfo $Id: quicksort-paths.tex fc32332 2016-05-10 11:11:47 git@danielkrenn.at $
\fi

\usepackage[utf8x]{inputenc}
\usepackage[T1]{fontenc}
\usepackage{lmodern}
\usepackage[english]{babel}
\usepackage[draft=false,kerning=true]{microtype}
\usepackage{units}

\PrerenderUnicode{ü}

\makeatletter
\let\DMTCSproof\proof
\let\DMTCSendproof\endproof
\let\proof\@undefined
\let\endproof\@undefined
\makeatother

\usepackage{amsmath, amsfonts, amssymb, amsthm, amstext, bbm}

\let\proof\DMTCSproof
\let\endproof\DMTCSendproof

\usepackage{tikz}
\usetikzlibrary{calc,shapes}
\usepackage{pgfplots}

\usepackage{ifthen,calc}
\usepackage{suffix}
\usepackage{enumerate}
\usepackage{nameref}
\usepackage{hyperref}

\usepackage{algorithm}
\usepackage[noend]{algpseudocode}

\newcommand{\parentheses}[4][]%
{\mathopen{}\ifthenelse{\equal{#1}{}}{\left#2}{\csname#1\endcsname#2}%
    {#4}\mathclose{}\ifthenelse{\equal{#1}{}}{\right#3}{\csname#1\endcsname#3}}
\newcommand{\foperator}[1]{\mathop{{#1}\empty{}}}
\newcommand{\f}[3][]{\ensuremath{\foperator{#2}\parentheses[#1]{(}{)}{#3}}}

\makeatletter
\newcommand\undisp[1]{\bgroup\@displayfalse #1\egroup}
\makeatother

\newcommand{\iverson}[1]{\ensuremath{\parentheses{[}{]}{#1}}}
\WithSuffix\newcommand{\iverson}*[1]{\ensuremath{\iverson{\text{\normalfont #1}}}}
\newcommand{\abs}[2][]{\ensuremath{%
    \parentheses[#1]{\lvert}{\rvert}{#2}}}
\newcommand{\floor}[1]{\ensuremath{\parentheses{\lfloor}{\rfloor}{#1}}}
\newcommand{\ceil}[1]{\ensuremath{\parentheses{\lceil}{\rceil}{#1}}}

\newcommand{\Ohsymbol}{O}
\newcommand{\Oh}[2][]{\ensuremath{\f[#1]{\Ohsymbol}{#2}}}
\newcommand{\set}[1]{\ensuremath{\parentheses{\{}{\}}{#1}}}
\WithSuffix\newcommand{\set}*[2]{\ensuremath{%
    \left\{{#1}\thinspace\setmiddlesymbol\thinspace{#2}\right\}}}

\newcommand{\calD}{\mathcal{D}}
\newcommand{\E}[2][]{\f[#1]{\mathbb{E}}{#2}}
\renewcommand{\P}[2][]{\f[#1]{\mathbb{P}}{#2}}

\DeclareMathOperator{\artanh}{artanh}
\newcommand{\Acv}{A^{\mathrm{cv}}}
\newcommand{\Act}{A^{\mathrm{ct}}}
\newcommand{\Ccv}{C^{\mathrm{cv}}}
\newcommand{\Cct}{C^{\mathrm{ct}}}
\newcommand{\Pcv}{P^{\mathrm{cv}}}
\newcommand{\Pct}{P^{\mathrm{ct}}}

\newcommand{\N}{\ensuremath{\mathbbm{N}}}
\newcommand{\Z}{\ensuremath{\mathbbm{Z}}}

\newcommand{\eps}{\ensuremath{\varepsilon}}
\newcommand{\Hodd}{H^{\mathrm{odd}}}
\newcommand{\Halt}{H^{\mathrm{alt}}}

\iffinal
\newcommand{\TODO}[1]{}
\else
\newcommand{\TODO}[1]
{\par\fbox{\begin{minipage}{0.9\linewidth}\textbf{TODO:} #1\end{minipage}}\par}
\fi

\newtheorem{thms}{Thm}[section] 
\theoremstyle{plain}
\newtheorem*{theorem*}{Theorem}
\newtheorem{theorem}[thms]{Theorem}
\newtheorem{lemma}[thms]{Lemma}
\newtheorem{corollary}[thms]{Corollary}
\newtheorem{proposition}[thms]{Proposition}
\theoremstyle{definition}

\theoremstyle{remark}
\newtheorem{remark}[thms]{Remark}
\newtheorem{corthm}[thms]{Corollary and Theorem}

\newtheoremstyle{strategy}{}{}{\itshape}{}{\bfseries}{}{.5em}{\thmname{#1} \thmnote{#3}. }
\theoremstyle{strategy}
\newtheorem*{strategy}{Strategy}

\makeatletter
\renewcommand{\part}[1]{%
  \@startsection{part}{0}{\z@}{\p@}{\p@}{\@gobble}{}%
  \section*{#1}}
\makeatother

\numberwithin{equation}{section}
\numberwithin{figure}{section}

\input{quicksort-paths.appendix-aofa.aux}

\title[Zeros in Random Walks on Integers and Dual-Pivot Quicksort]{Counting Zeros in Random Walks on the Integers and Analysis of Optimal Dual-Pivot Quicksort}

\author[M.~Aumüller, M.~Dietzfelbinger, C.~Heuberger, D.~Krenn, H.~Prodinger]{%
  Martin Aumüller\addressmark{1} \and
  Martin Dietzfelbinger\addressmark{2} \and
  Clemens Heuberger\addressmark{3}\footnotemark[3]%
  \footnotetext[3]{C.~Heuberger and D.~Krenn are supported by the
   Austrian Science Fund (FWF): P\,24644-N26
   and by the Karl Popper Kolleg ``Modeling--Simulation--Optimization''
   funded by the Alpen-Adria-Universität Klagenfurt and
   by the Carinthian Economic Promotion Fund (KWF).} \and \\
  Daniel Krenn\addressmark{4}\footnotemark[3] \and
  Helmut Prodinger\addressmark{5}\footnotemark[4]
  \footnotetext[4]{H.~Prodinger is supported by an incentive grant of the National
   Research Foundation of South Africa.}}

\address{%
  \addressmark{1}IT University of Copenhagen, 
  Denmark,
  \href{mailto:maau@itu.dk}{maau@itu.dk}\\
  \addressmark{2}Institut für Theoretische Informatik,
  Technische Universität Ilmenau, Germany,\\
  \href{mailto:martin.dietzfelbinger@tu-ilmenau.de}{martin.dietzfelbinger@tu-ilmenau.de}\\
  \addressmark{3}Institut f\"ur Mathematik,
  Alpen-Adria-Universit\"at Klagenfurt,
  Austria,
  \href{mailto:clemens.heuberger@aau.at}{clemens.heuberger@aau.at}\\
  \addressmark{4}Institut f\"ur Mathematik,
  Alpen-Adria-Universit\"at Klagenfurt,
  Austria,
  \href{mailto:math@danielkrenn.at}{math@danielkrenn.at} \textit{or}\\
  \href{mailto:daniel.krenn@aau.at}{daniel.krenn@aau.at}\\
  \addressmark{5}Department of Mathematical Sciences,
  Stellenbosch University, South Africa,
  \href{mailto:hproding@sun.ac.za}{hproding@sun.ac.za}}


\begin{document}
\maketitle

\begin{abstract}

	
	We present an average case analysis of two variants of dual-pivot
  quicksort, one with a non-algorithmic comparison-optimal partitioning strategy,
	the other with a closely related algorithmic strategy.
  For both we calculate the expected number of comparisons exactly as well as
  asymptotically, in particular, we provide exact expressions for the
  linear, logarithmic, and constant terms.
	An essential step is the analysis of zeros of lattice paths in a
  certain probability model. Along the way a combinatorial identity is
  proven.
\end{abstract}


\section{Introduction}
\label{sec:intro}


Dual-pivot quicksort~\cite{Sedgewick:1975:thesis,WildNN15,AumullerD15} 
is a family of sorting algorithms related to the
well-known quicksort algorithm. 
In order to sort an input sequence $(a_1,\ldots,a_n)$
of distinct elements, dual-pivot quicksort algorithms work as follows. 
(For simplicity we forbid repeated elements in the input.)
If $n \leq 1$, there is nothing to do.
If $n\ge2$, two 
input elements are selected as pivots.
Let $p$ be the smaller and $q$ be the larger pivot. 
The next step is to partition the remaining elements into
\begin{itemize}
\setlength{\itemsep}{1pt}
\setlength{\parskip}{0pt}
\setlength{\parsep}{0pt}
\item the elements smaller than $p$ (``small elements''),
\item the elements between $p$ and $q$ (``medium elements''), and
\item the elements larger than $q$ (``large elements'').
\end{itemize}
Then the procedure is applied recursively to these three groups to complete the sorting. 

The cost measure used in this work is the number of comparisons between elements.
As is common, we will assume the input sequence is in random order, 
which means that each permutation of the $n$ elements occurs with
probability $\nicefrac{1}{n!}$\,.
With this assumption we may, without loss of generality, choose $a_1$
and $a_n$ as the pivots. 
Even in this setting there are different dual-pivot quicksort algorithms;
their difference lies in the way the partitioning is organized,
which influences the partitioning cost. 
This is in contrast to standard quicksort with one pivot, where the
partitioning strategy does not influence the cost---in partitioning always
one comparison is needed per non-pivot element.
In dual-pivot quicksort, the average cost (over all permutations) of partitioning
and of sorting can be analyzed only when the partitioning strategy is fixed.

Only in 2009, Yaroslavskiy, Bentley, and Bloch~\cite{Yaroslavskiy-Mailinglist} described a dual-pivot
quicksort algorithm that makes $1.9 n \log n + O(n)$ comparisons~\cite{WildNN15}.%
\footnote{In this paper ``$\log$'' denotes the natural logarithm to base $e$.}
This beats the classical quicksort algorithm~\cite{Hoare62}, which needs
$2n \log n + O(n)$ comparisons on average. In
\cite{AumullerD15}, the first two authors of this article described
the full design space for dual-pivot quicksort algorithms with respect to
counting element comparisons.
Among others, they studied two special partitioning 
strategies. The first one---we call it ``Clairvoyant'' 
in this work---assumes that the number of small and large elements
is given (by an ``oracle'') before partitioning starts.
It cannot be implemented; however, it is optimal among all partioning strategies
that have access to such an oracle, and hence its cost provides a lower bound for 
the cost of all algorithmic partitioning strategies.
In~\cite{AumullerD15} it was shown that dual-pivot quicksort carries out $1.8 n \log n + O(n)$
comparisons on average when this partitioning strategy is used. 
Further a closely related algorithmic partitioning 
strategy---called ``Count'' here---was described, which makes only $O(\log n)$ more
comparisons on average than ``Clairvoyant'' and hence leads to 
a dual-quicksort variant with only $O(n)$ more comparisons.\footnote{After
  completing this extended abstract we
  found a proof that ``Count'' is optimal among all algorithmic
  strategies. Details to be given in the full version.}

One purpose of this paper is to make the expected number of
comparisons in both variants precise and to determine the exact difference 
of the cost of these two strategies, both for partitioning and for
the resulting dual-pivot quicksort variants.

Already in \cite{AumullerD15} it was noted that 
the exact value of the expected partitioning cost
(i.e., the number of comparisons) of both strategies
depends on the expected number of the zeros of certain lattice paths
(Parts~\ref{sec:lattice-paths} and~\ref{sec:more-lattice-paths}).
A complete understanding of this situation is the basis for our analysis
of dual-pivot quicksort, which appears in Part~\ref{sec:quicksort}.

Lattice path enumeration has a long tradition. An early reference is
\cite{Mohanty:1979:lattic};
a recent survey paper is \cite{Krattenthaler:2015:lattic}. 
As space is limited, many proofs and some additional results can be found in
an appendix at \href{http://arxiv.org/abs/1602.04031v1}{arXiv:1602.04031v1}.


\section{Overview and Results}
\label{sec:results}


This work is split into three parts. We give a brief overview on the
main results of each of these parts here.
We use the Iversonian expression
\begin{equation*}
  \iverson{\mathit{expr}} =
  \begin{cases}
    1&\text{ if $\mathit{expr}$ is true},\\
    0&\text{ if $\mathit{expr}$ is false},
  \end{cases}
\end{equation*}
popularized by Graham, Knuth, and Patashnik~\cite{Graham-Knuth-Patashnik:1994}.

The harmonic numbers and their variants will be denoted by
\begin{equation*}
  H_n = \sum_{m=1}^n \frac{1}{m}, \qquad
  \Hodd_n = \sum_{m=1}^n \frac{\iverson*{$m$ odd}}{m}
  \qquad\text{and}\qquad
  \Halt_n = \sum_{m=1}^n \frac{(-1)^m}{m}.
\end{equation*}
Of course, there are relations between these three definitions such as
$\Halt_n=H_n - 2\Hodd_n$ and $\Hodd_n + H_{\floor{n/2}}/2=H_n$, but it will
turn out to be much more convenient to use all three notations.


\subsection*{\nameref{sec:lattice-paths}}

In the first part we analyze certain lattice paths of a fixed
length~$n$. 
We start on the vertical axis,
allow steps/increments $(1,+1)$ and $(1,-1)$ and
end on the horizontal axis at $(n,0)$.
To be precise, the starting point on the vertical axis
is chosen uniformly at random from the set $\set{(0,-n),(0,-n+2),\dots,(0,n-2),(0,n)}$
of feasible points.
Once this starting point is
fixed, all paths to $(n,0)$ are equally likely. We are interested in
the number of zeros, denoted by the random
variable~$Z_n$, of such paths.

An exact formula for the expected number $\E{Z_n}$ of zeros is derived
in two different ways (see identity~\eqref{eq:id-intro} for these formul\ae):
On the one hand, we use the symbolic method and
generating functions (see Appendix~\ref*{sec:gf}), which gives the
result in form of a double sum.
This machinery extends well to higher moments and also allows us to
obtain the distribution. The exact distribution is given in
Appendix~\ref*{sec:distribution}; its limiting behavior as $n\to\infty$ is
the discrete distribution
\begin{equation*}
  \P{Z_n = r} \sim \frac{1}{r(r+1)}.
\end{equation*}
On the other hand, a more probabilistic approach gives the 
expectation~$\E{Z_n}$ as the simple single sum
\begin{equation*}
  \E{Z_n} = \sum_{m=1}^{n+1} \frac{\iverson*{$m$ odd}}{m}=\Hodd_{n+1},
\end{equation*}
see Section~\ref{sec:prob} for more details. The
asymptotic behavior $\E{Z_n} \sim \frac12 \log n$ can be extracted
(Appendix~\ref*{sec:asymptotics}).

The two approaches above give rise to the identity
\begin{equation}\label{eq:id-intro}
  \sum_{m=1}^{n+1} \frac{\iverson*{$m$ odd}}{m}
  =
  \frac{4}{n+1}
  \sum_{0\leq k < \ell < \ceil{n/2}} \frac{\binom{n}{k}}{\binom{n}{\ell}}
  + \iverson*{$n$ even} \frac{1}{n+1}
  \biggl(\frac{2^n}{\binom{n}{n/2}} - 1\biggr) + 1;
\end{equation}
the double sum above equals the single sum of
Theorem~\ref{thm:paths-prob} by combinatorial considerations. One
might ask about a direct proof of this identity. This can be achieved
by methods related to hypergeometric sums and the computational proof
is presented in Appendix~\ref*{sec:identity}. We also provide a completely
elementary proof which is ``purely human''.


\subsection*{\nameref{sec:more-lattice-paths}}

The second part acts as connecting link between the lattice paths of
fixed length of Part~\ref{sec:lattice-paths} and the dual-pivot
quicksort algorithms of Part~\ref{sec:quicksort}.

The probabilistic model introduced in Section~\ref{sec:description}
(in Part~\ref{sec:lattice-paths}) is extended, and lattice paths are
allowed to vary in length. For a number~$n$ (the number of
elements to sort) the length of a path is the number of elements remaining
when the two pivots, given by a random set of elements of size two, 
and the elements between these pivots are cut out.

The number of zeros~$X_n$ in this full model is analyzed; we provide again
exact as well as asymptotic formul\ae{} for the
expectation~$\E{X_n}$. Details are given in
Section~\ref{sec:lattice-paths-N}. Moreover, more specialized
zero-configurations (needed for the analysis of different partitioning
strategies in Part~\ref{sec:quicksort}) are considered as well
(Section~\ref{sec:more-zeros}).


\subsection*{\nameref{sec:quicksort}}

The main result of this work analyzes comparisons in the dual-pivot
quicksort algorithm that uses the optimal (but unrealistic) partitioning strategy
``Clairvoyant''. Aumüller and Dietzfelbinger showed in \cite{AumullerD15} that this algorithm requires 
$1.8 n \log n + O(n)$ comparisons on average, which improves on 
the average number of comparisons in quicksort ($2n \log n + O(n)$) and the 
recent dual-pivot algorithm of Yaroslavskiy et al.\@ ($1.9 n \log n + O(n)$, see \cite{WildNN15}). However, for real-world input sizes $n$ the (usually negative) factor in the linear term has a great influence on the comparison count. Our asymptotic result is stated as the following theorem.

\begin{theorem*}
  The average number of comparisons in the dual-pivot quicksort
  algorithm with a comparison-optimal partitioning strategy is
  \begin{equation*}
    \frac{9}{5} n \log n + A n + B \log n + C + \Oh{1/n}
  \end{equation*}
  as $n\to\infty$, with
    $A = \frac95\gamma
    - \frac{1}{5} \log 2
    - \frac{89}{25}
    = -2.659\dots$\,.
\end{theorem*}

The constants $B$ and $C$ are explicitly given, too, and more
terms of the asymptotics are presented. The precise result is formulated
as Corollary~\ref{cor:clairvoyant:cost:asy}.

In fact, we even get an exact expression for the average comparison
count. The precise result is formulated as
Theorem~\ref{thm:clairvoyant:cost}.  Moreover the same analysis is
carried out for the partitioning strategy ``Count'', which
is an algorithmic variant of the comparison-optimal strategy
``Clairvoyant''. Aumüller and Dietzfelbinger \cite{AumullerD15} could show that it
requires $\frac95n \log n + O(n)$ comparisons as well. In this paper
we obtain the exact average comparison count
(Theorem~\ref{thm:count:cost}). The asymptotic result is again
$\frac{9}{5} n \log n + A n + \Oh{\log n}$,
but now with $A = -2.382\dots$, so there is only a small gap
between the average number of comparisons in the comparison-optimal
strategy ``Clairvoyant'' and its algorithmic variant.



\part{Part I: Lattice Paths}
\label{sec:lattice-paths}


In this first part we analyze lattice paths of a fixed
length~$n$. These are introduced in Section~\ref{sec:description} by a
precise description of our probabilistic model. We will work with this
model throughout Part~\ref{sec:lattice-paths}, and we analyze the number
of zeros~$Z_n$.

The outline is as follows: We derive an exact expression for
the expected number $\E{Z_n}$ of zeros by the generating functions
machinery in Appendix~\ref*{sec:gf}; a more
probabilistic approach can be found in Section~\ref{sec:prob}.
Appendix~\ref*{sec:asymptotics} deals with asymptotic considerations. Direct
proofs of the obtained identity are given in
Appendix~\ref*{sec:identity} and the distribution of $Z_n$ is tackled in Appendix~\ref*{sec:distribution}.


\section{Probabilistic Model}
\label{sec:description}


We consider paths of a given length~$n$ on the lattice $\Z^2$, where only steps
$(1,+1)$ and $(1,-1)$ are allowed. These paths are chosen at random according to the rules
below.

Let us fix a length~$n\in\N_0$. A path~$P_n$ ending in $(n,0)$ (no
choice for this end-point) is chosen according to the following rules.
\begin{enumerate}
\item First, choose a starting point $(0,S)$ where $S$ is a random
  integer uniformly distributed in $\{-n, -n+2, \ldots, n-2, n\}$, i.\,e., $S=s$ occurs only for
  integers $s$ with $\abs{s}\le n$ and $s\equiv n \pmod 2$.
\item Second, a path is chosen uniformly at random among all paths from $(0,S)$
  to $(n,0)$.
\end{enumerate}
 The conditions on $S$ characterize those starting points from which $(n, 0)$ is reachable.

We are interested in the number of intersections with the horizontal axis of
a path. To make this precise, we define a \emph{zero} of a path $P_n$ as a point
$(x,0)\in P_n$. 

Thus, let $P_n$ be a path of length~$n$ which is chosen according to the
probabilistic model above and define the random variable
\begin{equation*}
  Z_n = \text{number of zeros of $P_n$}.
\end{equation*}

In the following sections, we determine the value of $\E{Z_n}$ exactly
(Appendix~\ref*{sec:gf} and Section~\ref{sec:prob}), as well
as asymptotically (Appendix~\ref*{sec:asymptotics}). In Appendix~\ref*{sec:gf}, we
use the machinery of generating functions. This machinery turns out to be
overkill if we are just interested in the expectation $\E{Z_n}$. However, it
easily allows extension to higher moments and the limiting distribution.

In Section~\ref{sec:prob}, we follow a
probabilistic approach, which first gives a result on
the probability model that at the first glance looks surprising:
the equidistribution at the initial values turns out to
carry over to every fixed length of the remaining path. This result yields
a simple expression for the expectation
$\E{Z_n}$ in terms of harmonic numbers, and
thus immediately yields a precise asymptotic expansion for $\E{Z_n}$.
The generating function approach,
however, gives the expectation in terms of a double sum of quotients of
binomial coefficients (the right-hand side of \eqref{eq:id-intro}),
see Appendix~\ref*{sec:gf}.

Appendix~\ref*{sec:identity} gives a direct computational
proof that these two results coincide. The original expression in
\cite{AumullerD15} (a double sum over a quotient
of a product of binomial coefficients and a binomial coefficient) is also shown
to be equal in Appendix~\ref*{sec:identity}.
Explicit as well as asymptotic expressions for the distribution
$\P{Z_n=r}$ can be found in Appendix~\ref*{sec:distribution}.


\def\appendixgf{
\section{Using the Generating Function Machinery}
\label{sec:gf}


\begin{theorem}\label{thm:paths-gf-zeros}
  For a randomly (as described in Section~\ref{sec:description}) chosen path of
  length~$n$, the expected number of zeros is
  \begin{equation*}
    \E{Z_n} =
    \frac{4}{n+1}
    \sum_{0\leq k < \ell < \ceil{n/2}} \frac{\binom{n}{k}}{\binom{n}{\ell}}
    + \iverson*{$n$ even} \frac{1}{n+1}
    \left(\frac{2^n}{\binom{n}{n/2}} - 1\right) + 1.
  \end{equation*}
\end{theorem}

The remaining part of this section is devoted to the proof of this theorem.
The main technique is to model our lattice paths by means of combinatorial
classes and generating functions. For simplicity of notation,
we denote a combinatorial class by its corresponding generating function.

\begin{figure}
  \centering
    \begin{tikzpicture}[scale=0.25, latticepath/.style={very thick}]

    \draw (-22,-3) -- (-22,14);
    \draw (-24,0) -- (25,0);

    \newcommand{\Cpath}[3][]{
      \draw[latticepath, #1] ($(0,0) + #2$) -- ($(4,0) + #2$) --
      ($(2,2.82842712474619) + #2$) -- cycle;
      \node at ($(2,0) + #2$) [above] {#3};
    }
    \newcommand{\Cpathmirr}[3][]{
      \draw[latticepath, #1] ($(0,0) + #2$) -- ($(4,0) + #2$) --
      ($(2,-2.82842712474619) + #2$) -- cycle;
      \node at ($(2,0) + #2$) [below] {#3};      
    }

    \newcommand{\Cpathdown}[2]{
      \Cpath{#1}{#2}
      \draw[latticepath] ($(4,0) + #1$) -- ($(5,-1) + #1$);
    }

    \node at (-22,10) [left] {$s$};
    \Cpathdown{(-22,10)}{$C$}
    \Cpathdown{(-17,9)}{$C$}

    \draw[dotted] (-11,7) -- (-7,3);
    \draw[latticepath] (-6,2) -- (-5,1);
    \Cpathdown{(-5,1)}{$C$}

    \node[circle,inner sep=1.5pt,fill] at (0,0) {};
    \node[] at (0,0) [below] {$u$};
    \draw[latticepath] (0,0) -- (1,1);
    \Cpath{(1,1)}{$C$}
    \draw[latticepath] (5,1) -- (6,0);
    \draw[latticepath] (0,0) -- (1,-1);    
    \Cpathmirr{(1,-1)}{$C$}
    \draw[latticepath] (5,-1) -- (6,0);

    \node[circle,inner sep=1.5pt,fill] at (6,0) {};
    \node[] at (6,0) [below] {$u$};
    \draw[latticepath] (6,0) -- (7,1);
    \Cpath{(7,1)}{$C$}
    \draw[latticepath] (11,1) -- (12,0);
    \draw[latticepath] (6,0) -- (7,-1);    
    \Cpathmirr{(7,-1)}{$C$}
    \draw[latticepath] (11,-1) -- (12,0);

    \node[circle,inner sep=1.5pt,fill] at (12,0) {};
    \node[] at (12,0) [below] {$u$};
    \draw[dotted] (12,2.414) -- (18,2.414);
    \draw[dotted] (12,-2.414) -- (18,-2.414);

    \node[circle,inner sep=1.5pt,fill] at (18,0) {};
    \node[] at (18,0) [below] {$u$};
    \draw[latticepath] (18,0) -- (19,1);
    \Cpath{(19,1)}{$C$}
    \draw[latticepath] (23,1) -- (24,0);
    \draw[latticepath] (18,0) -- (19,-1);    
    \Cpathmirr{(19,-1)}{$C$}
    \draw[latticepath] (23,-1) -- (24,0);

    \node at (24,0) [below right] {$n$};

  \end{tikzpicture}

  \caption{Decomposition of a lattice path for $s\geq0$ marking zeros.}
  \label{fig:decomp-path-zeros}
\end{figure}

Concerning the generating functions, we mark a step to the right by
the variable~$z$ and a zero (except the last) by~$u$. Note that we do
not mark the zero at $(n,0)$ for technical reasons; we'll take this
into account at the end by adding a $1$ to the final result.
Thus, the coefficient of $z^nu^{r-1}$ of
the function $Q_s(z,u)$ (the generating function of all paths starting in
$(0,s)$ and ending in some $(n,0)$) equals the number of paths of length~$n$
and exactly $r$ zeros.

We also need the following auxiliary function. The generating
function~$\f{C}{z}$ counts all Catalan paths, i.\,e., paths starting and ending
at the same height, but not going below it. This equals
\begin{equation*}
  \f{C}{z} = \frac{1-\sqrt{1-4z^2}}{2z^2}.
\end{equation*}

In Figure~\ref{fig:decomp-path-zeros}, we give a schematic decomposition of a
path from $(0, s)$ to $(n, 0)$ for non-negative $s$. This decomposition
translates to the following parts of the generating function (the path is read
from the left to the right).
\begin{itemize}
\item We start by $s$ consecutive blocks of $\f{C}{z}$, each followed by a single
  descent encoded as $z$. This gives the paths from $(0,s)$ to their first
  zero (i.\,e., where it touches the horizontal axis for the first time).
\item We mark this zero by the symbol~$u$.
\item We either do a single ascent or a single decent (marked by a~$z$), then
  continue with a $\f{C}{z}$-block and do a single decent or ascent
  respectively (marked by a~$z$ as well) again. Thus, we are back at a zero.
\item We repeat such up/down blocks $z\f{C}{z}z$, each one preceded by a
  zero~$u$, a finite number of times.
\end{itemize}
If $s<0$, then the construction is the same, but everything is reflected at the
horizontal axis.

Continuing using the symbolic method---the description above is already part of
it, see, for example, Flajolet and
Sedgewick~\cite{Flajolet-Sedgewick:ta:analy}---the decomposition above
translates to the generating function
\begin{equation}
  \label{eq:path-gf-zeros}
  Q_s(z,u) = \frac{\f{C}{z}^{\abs{s}} z^{\abs{s}}}{1 - 2 u z^2 \f{C}{z}},
\end{equation}
which we will use from now on. Note that the coefficient $2$ reflects the fact
that there are two choices (up and down) for the blocks between zeros.

To obtain a nice explicit form, we perform a change of variables. The result is
stated in the following lemma.

\begin{lemma}\label{lem:transform-gf-zeros}
  With the transformation $z = v / (1+v^2)$ we have
  \begin{equation*}
    Q_s(z,u) = \frac{v^{\abs{s}}(1+v^2)}{1-v^2(2u-1)}.
  \end{equation*}
\end{lemma}

\begin{proof}
  Transforming the counting generating function of Catalan paths yields
  \begin{equation*}
    \f{C}{z} = 1+v^2.
  \end{equation*}
  Thus~\eqref{eq:path-gf-zeros} becomes
  \begin{equation*}
    Q_s(z,u) = 
    (1+v^2)^{\abs{s}}
    \Big(\frac{v}{1+v^2}\Big)^{\abs{s}}
    \frac{1}{1-2u\big(\frac{v}{1+v^2}\big)^2(1+v^2)}
  \end{equation*}
  and can be simplified to the expression stated in the lemma.
\end{proof}

The next step is to extract the coefficients out of the expressions obtained in
the previous lemma. First we rewrite the extraction of the coefficients from
the ``$z$-world'' to the ``$v$-world'', see
Lemma~\ref{lem:extract-coeffs-worlds}. Afterwards, in
Lemma~\ref{lem:coeffs-zeros}, the coefficients can be determined quite easily.

\begin{lemma}\label{lem:extract-coeffs-worlds}
  Let $F(z)$ be an analytic function in a neighborhood of the origin. Then we have
  \begin{equation*}
    [z^n] F(z) = [v^n] (1-v^2) (1+v^2)^{n-1}
      \f{F}{\frac{v}{1+v^2}}.
  \end{equation*}
\end{lemma}

\begin{proof}
  We use Cauchy's formula to extract the coefficients of $F(z)$ as
  \begin{equation*}
    [z^n] F(z) = \frac{1}{2\pi i}\oint_\calD \frac{dz}{z^{n+1}} F(z)
  \end{equation*}
  where $\calD$ is a positively oriented small circle around the origin. Under
  the transformation $z=v/(1+v^2)$, the circle $\calD$ is transformed to a
  contour $\calD'$ which still winds exactly once around the origin. Using
  Cauchy's formula again, we obtain
  \begin{align*}
    [z^n] F(z)
    &= \frac{1}{2\pi i}\oint_{\calD'} \frac{dv(1-v^2)}{(1+v^2)^2}
    \frac{(1+v^2)^{n+1}}{v^{n+1}} \f{F}{\frac{v}{1+v^2}} \\
    &= [v^n] (1-v^2) (1+v^2)^{n-1} \f{F}{\frac{v}{1+v^2}}.
  \end{align*}
\end{proof}

Now we are ready to calculate the desired coefficients.

\begin{lemma}\label{lem:coeffs-zeros}
  Suppose $n\equiv s\pmod 2$. Then we have
  \begin{equation*}
    [z^n] Q_s(z,1)
    = \binom{n}{(n-s)/2}
  \end{equation*}
  and, moreover,
  \begin{equation*}
    [z^n] \left.\frac{\partial}{\partial u} Q_s(z,u)\right\vert_{u=1}
    = 2 \sum_{k=0}^{(n-\abs{s})/2-1}\binom{n}{k}.
  \end{equation*}
\end{lemma}

\begin{proof}
  As $n \equiv s \pmod 2$, the number $n-s$ is even, and so we can
  set $\ell = \frac12(n-s)$. Then $[z^n]Q_s(z, 1)$ is the number of paths from
  $(0, s)$ to $(n, 0)$. These paths have $\ell$ up steps and $n-\ell$ down
  steps; thus there are $\binom{n}{\ell}$ many such paths.

  For the second part of this lemma, we restrict ourselves to $s\geq0$ (otherwise use $-s$ and the symmetry in~$s$
  of the generating function~\eqref{eq:path-gf-zeros} instead). We start with the result of
  Lemma~\ref{lem:transform-gf-zeros}. Taking the first derivative and setting
  $u=1$ yields
  \begin{equation*}
    \left.\frac{\partial}{\partial u} Q_s(z,u)\right\vert_{u=1}
    = \frac{2v^{s+2}(1+v^2)}{(1-v^2)^2}.
  \end{equation*}
  Thus, by using Lemma~\ref{lem:extract-coeffs-worlds}, we get
  \begin{equation*}
    [z^n] \frac{2v^{s+2}(1+v^2)}{(1-v^2)^2}
    = 2\,[v^{n-s-2}]\frac{(1+v^2)^{n}}{1-v^2}.
  \end{equation*}
  We use $\ell$ as above and get
  \begin{equation*}
    [v^{n-s-2}]\frac{(1+v^2)^n}{1-v^2}
    = [v^{2\ell-2}]\frac{(1+v^2)^n}{1-v^2}
    = [v^{\ell-1}] \frac{(1+v)^n}{1-v}
    = \sum_{k=0}^{\ell-1}\binom{n}{k},
  \end{equation*}
  which was claimed to hold.
\end{proof}

We are now ready to prove the main theorem (Theorem~\ref{thm:paths-gf-zeros})
of this section, which provides an expression for the expected number of
zeros. This exact expression is written as a double sum.

\begin{proof}[of Theorem~\ref{thm:paths-gf-zeros}]
  By Lemma~\ref{lem:coeffs-zeros}, the average number of zeros
  (except the zero at the end point) of
  a path of length~$n$ which starts in $(0,s)$ is
  \begin{equation*}
    \mu_{n,s} = \frac{[z^n] \left.\frac{\partial}{\partial u}
        Q_s(z,u)\right\vert_{u=1}}{[z^n] Q_s(z,1)}
    = \frac{2}{\binom{n}{\ell}} \sum_{k=0}^{\ell-1}\binom{n}{k},
  \end{equation*}
  where we have set $\ell = \frac12 (n-\abs{s})$ as in the proof of
  Lemma~\ref{lem:coeffs-zeros}. If $s=0$, this simplifies to
  \begin{equation}\label{eq:mu_n_0}
    \mu_{n,0} = \frac{2}{\binom{n}{n/2}} \sum_{k=0}^{n/2-1}\binom{n}{k}
    = \frac{2^n}{\binom{n}{n/2}} - 1.
  \end{equation}
  If $n\not\equiv s\pmod 2$, then we set $\mu_{n,s} = 0$.

  Summing up yields
  \begin{align*}
    \sum_{s=-n}^n \mu_{n,s}
    &= 2 \sum_{s=1}^n \mu_{n,s} + \mu_{n,0}
    = 4 \sum_{\ell=0}^{\ceil{n/2}-1}
    \frac{1}{\binom{n}{\ell}} \sum_{k=0}^{\ell-1}\binom{n}{k}
    + \mu_{n,0} \\
    &= 4 \sum_{0\leq k < \ell < \ceil{n/2}} \frac{\binom{n}{k}}{\binom{n}{\ell}}
    + \iverson*{$n$ even} 
    \left(\frac{2^n}{\binom{n}{n/2}} - 1\right).
  \end{align*}
  Dividing by the number $n+1$ of possible starting points and adding
  $1$ for the zero at $(n,0)$ completes the proof of
  Theorem~\ref{thm:paths-gf-zeros}.
\end{proof}
}


\section{A Probabilistic Approach}
\label{sec:prob}


\begin{theorem}\label{thm:paths-prob}
  For a randomly (as described in Section~\ref{sec:description}) chosen path of
  length~$n$, the expected number of zeros is
  \begin{equation*}
    \E{Z_n} = \Hodd_{n+1}.
  \end{equation*}
\end{theorem}

Before proving the theorem, we consider an equivalent probability model for our
random paths formulated as an urn model. A number $R$ from $\{0,\dots,n\}$ is chosen uniformly at random.
We place $R$ red balls and $B=n-R$ black balls in an urn.
Subsequently, in $n$ rounds the balls are taken from the urn (without replacements), in each round choosing one
uniformly at random. The color of the ball chosen in round $i$ is denoted by $U_i$.

We construct a random walk $(W_i)_{0\le i\le n}$ on $\{-n, \ldots, n\}$ from
$U_1$, \ldots, $U_n$ by setting $W_0 = R - B = 2R-n$ and
\begin{equation*}
  W_i =
  \begin{cases}
    W_{i-1} + 1&\text{ if $U_i=\mathrm{black}$},\\
    W_{i-1} - 1&\text{ if $U_i=\mathrm{red}$}
  \end{cases}
\end{equation*}
for $1\le i\le n$.
In each step, $W_i$ equals the difference of the number of remaining red and
black balls in the urn. Clearly, then, $W_n=0$.

One can look at the trajectories of this random walk, represented in
the grid $\{0,\ldots,n\}\times\{-n,\dots,n\}$ as
sequences $((0,W_0),(1,W_1),\dots,(n,W_n))$.
\def\appendixequrnpaths{
Only sequences with $W_0 \equiv n \pmod 2$ and
$W_{i} - W_{i-1}\in\{1,-1\}$ for $1\le i \le n$ and $W_n=0$ can have a probability bigger than 0.
We denote the set of these paths by ${\mathcal{P}}_n$.

Now fix $w_0$ with $|w_0|\le n$ and  $w_0\equiv n \pmod 2$.
Clearly, all sequences $p_n\in{\mathcal{P}}_n$ starting with $(0,w_0)$
together have probability $1/(n+1)$.
What is the probability of
a path $p_n=((0,w_0),(1,w_1),\dots,(n,w_n)) \in {\mathcal{P}}_n$
to appear as trajectory of the random walk?

The starting point $(0, w_0)$ fixes the number of black and red balls as
$r=(n+w_0)/2$ and $b=(n-w_0)/2$, respectively.
Let us number the $n$ balls arbitrarily with $1$, \ldots, $n$.
The random experiment in essence arranges the balls in
one of the $n!$ possible orders, each of these orders
having the same probability.
A trajectory is determined not by this order, but
by the pattern of red and black balls in the sequence.
Thus $r!\,b!$ orders lead to the same trajectory $p_n$.
Thus, given $w_0$, all $\binom{n}{r,b}$ paths starting in $(0, w_0)$ are equally likely.
Therefore this urn model is equivalent to the model given in Section~\ref{sec:description}.
}
Appendix~\ref*{sec:appendix:prob} explains the equivalence between
the two models.

In order to prove Theorem~\ref{thm:paths-prob}, we need the following
property of our paths.

\begin{lemma}\label{lem:prob-point-uniform}
  Let $m\in\N_0$ with $m\leq n$. The probability that a random path $P_n$ (as
  defined in Section~\ref{sec:description}) runs through $(n-m,k)$ is
  \begin{equation}\label{eq:probability}
    \P{(n-m, k)\in P_n} = \frac{1}{m+1}
  \end{equation}
  for all $k$ with $\abs{k}\leq m$ and
  $k\equiv m \pmod 2$, otherwise $0$.
\end{lemma}

The proof of this lemma can be found in Appendix~\ref*{sec:appendix:prob}.

A closer look reveals that when we reverse the paths, our model is
equivalent to a contagion P{\'o}lya urn model with two colors,
starting with one ball of each color, where we sample with replacement
and put another ball of the color just drawn into the urn. In this
setting, uniform distribution for feasible points with the same
first coordinate and hence the result of the lemma are well-known
phenomena. These results and more on the urn model can be found,
for example, in Mahmoud~\cite{Mahmoud:2008:Polya-urn-models}.

We continue with the actual proof of our theorem.

\def\appendixproofprobuniform{
\begin{proof}[of Lemma~\ref{lem:prob-point-uniform}]
  We compute the transition probabilities by considering the urn
  model. The probabilty of the event $U_{i+1}=\mathrm{black}$ is the number of
  remaining black balls divided by the number of remaining balls $n-i$. If
  $W_i=\ell$, then there are still $(n-i+\ell)/2$ red and $(n-i-\ell)/2$ black
  balls in the urn.

  Thus
  \begin{equation}\label{eq:transition-probabilities}
    \begin{aligned}
      \P{W_{i+1}=\ell+1 \mid W_i = \ell} &= \frac{n-i-\ell}{2(n-i)},\\
      \P{W_{i+1}=\ell-1 \mid W_i = \ell} &= \frac{n-i+\ell}{2(n-i)}.
    \end{aligned}
  \end{equation}
  We now prove the claim by backwards induction on $m$. For $m=n$, the
  assertion follows directly from the probability model.

  For $m<n$, we obtain
  \begin{align*}
    \P{(n-m, k)\in P_n}&=\P{W_{n-m}=k}\\
    &=\P{W_{n-m}=k\mid
      W_{n-m-1}=k-1}\P{W_{n-m-1}=k-1}\\
    &\phantom{=}\hphantom{0}+
    \P{W_{n-m}=k\mid
      W_{n-m-1}=k+1}\P{W_{n-m-1}=k+1}\\
    &=\frac{(m+1-(k-1))+(m+1+(k+1))}{(m+2)\,2(m+1)}=\frac{1}{m+1}
  \end{align*}
  by~\eqref{eq:transition-probabilities} and the induction hypothesis.
\end{proof}
}

\bgroup\def\endproof{}
\begin{proof}[of Theorem~\ref{thm:paths-prob}]
  By Lemma~\ref{lem:prob-point-uniform}, the expected number of zeros of
  $P_n$ is
  \begin{equation*}
    \E{Z_n} =
    \sum_{m=0}^{n} \P{(n-m,0) \in P_n}
    = \sum_{m=0}^{n} \frac{\iverson*{$m$ even}}{m+1}
    = \sum_{m=1}^{n+1} \frac{\iverson*{$m$ odd}}{m}
    = \Hodd_{n+1}.\tag*{\qed}
  \end{equation*}
\end{proof}
\egroup


\section{Additional Results}
\label{sec:additional}

The expected number of zeros can be evaluated asymptotically. We
obtain
\begin{corollary}
  \begin{equation*}
    \E{Z_n} =
    \frac12 \log n + \frac{\gamma+\log2}{2}
    + \frac{1+\iverson*{$n$ even}}{2n} - \frac{2+9\iverson*{$n$
        even}}{12n^2} + 
    \Oh[Big]{\frac{1}{n^3}}
  \end{equation*}
  asymptotically as $n$ tends to infinity.
\end{corollary}
The proof of this result uses the well-known asymptotic expansion of the harmonic numbers.
  The actual asymptotic computations\footnote{A worksheet containing
    the computations can be found at
    \url{http://www.danielkrenn.at/downloads/quicksort-paths/quicksort-paths.ipynb}.}
  have been carried out using the asymptotic
  expansions module \cite{Hackl-Krenn:2015:asy-sage} of
  SageMath~\cite{SageMath:2016:7.0}, see
  Appendix~\ref*{sec:asymptotics}.

By combining the generating function and probabilistic approach we
obtain the following identity.
\begin{theorem}
  For $n\ge 0$, we have
  \begin{align*}
    &\mathrel{\phantom{=}} \frac{4}{n+1} \sum_{0\leq k < \ell <
      \ceil{n/2}} \frac{\binom{n}{k}}{\binom{n}{\ell}} + \iverson*{$n$
      even} \frac{1}{n+1}
    \left(\frac{2^n}{\binom{n}{n/2}} - 1\right) + 1\\
    &= \frac{1}{n+1} \sum_{m=0}^{\floor{n/2}} \sum_{\ell=m}^{n-m}
    \frac{\binom{2m}{m} \binom{n-2m}{\ell-m}}{\binom{n}{\ell}}
    = \Hodd_{n+1}.
  \end{align*}
\end{theorem}
The second expression for the expected number of zeros, but without
taking the zero at $(n,0)$ into account, has been given in
\cite[displayed equation after (14)]{AumullerD15}.  In
Appendix~\ref*{sec:identity} we give two direct proofs of the identity
above: One of them follows a computer generated proof (``creative telescoping'')
by extracting the essential recurrence. The second proof is ``human'' and
completely elementary using not more than Vandermonde's convolution.

Furthermore, the generating function machinery allows us to determine the distribution of the 
number~$Z_n$ of zeros. Beside an exact formula (see Appendix~\ref*{sec:distribution}), we get the following asymptotic result.

\begin{theorem}
  Let $0<\eps\leq\frac12$. For positive integers $r$ with $r =
  \Oh{n^{1/2-\eps}}$, we have asymptotically
  \begin{equation*}
    \P{Z_n = r} = \frac{1}{r(r+1)} \left(1 + \Oh{1/n^{2\eps}}\right)
  \end{equation*}
  as $n$ tends to infinity.
\end{theorem}


\def\appendixidentity{
\section{Identity}
\label{sec:identity}


Using the two results of Section~\ref{sec:prob} and Appendix~\ref*{sec:gf} we can show the following identity in a
combinatorial way.

\addtocounter{equation}{-1}
\begin{corthm}\label{thm:identity}
  For $n\ge 0$, we have the four equal expressions
  \begin{subequations}
  \renewcommand{\theequation}{\roman{equation}}
  \begin{align}
    &\mathrel{\phantom{=}}
    \frac{4}{n+1}
    \sum_{0\leq k < \ell < \ceil{n/2}} \frac{\binom{n}{k}}{\binom{n}{\ell}}
    + \iverson*{$n$ even} \frac{1}{n+1}
    \left(\frac{2^n}{\binom{n}{n/2}} - 1\right) + 1
    \label{eq:id:double}\\
    &= \frac{2}{\floor{n/2}+1}
      \sum_{0\leq k < \ell \leq \floor{n/2}}
      \frac{\binom{2\floor{n/2}+1}{k}}{\binom{2\floor{n/2}+1}{\ell}} + 1
      \label{eq:id:double-simple}\\
    &= \frac{1}{n+1} \sum_{m=0}^{\floor{n/2}} \sum_{\ell=m}^{n-m}
    \frac{\binom{2m}{m} \binom{n-2m}{\ell-m}}{\binom{n}{\ell}}
    \label{eq:id:quicksort}\\
    &= \Hodd_{n+1}.
    \label{eq:id:single}
  \end{align}
  \end{subequations}
\end{corthm}

We note that the expressions \eqref{eq:id:double} and
  \eqref{eq:id:double-simple} are obviously equal for odd $n$. Furthermore,
  \eqref{eq:id:double-simple} and
  \eqref{eq:id:single} only change when $n$ increases by $2$
  (to be precise from odd~$n$ to even~$n$).
Once we prove that \eqref{eq:id:double}
equals \eqref{eq:id:single} for all $n\ge 0$, it follows that both expressions
are equal to \eqref{eq:id:double-simple} for all $n\ge 0$.

\begin{proof}[(Combinatorial proof of Corollary~\ref{thm:identity})]
  Combine the results of Theorems~\ref{thm:paths-gf-zeros}
  and~\ref{thm:paths-prob} to see that \eqref{eq:id:double} and \eqref{eq:id:single}
  are equal.

  Expression~\eqref{eq:id:quicksort} for the expected number of zeros,
  but without taking the zero at $(n,0)$ into account,
  has been given in \cite[displayed
  equation after (14)]{AumullerD15}: The number of
  paths from $(0, n-2\ell)$ to $(n, 0)$ is $\binom{n}{\ell}$, whereas the number of
  paths from $(0, n-2\ell)$ via $(n-2m, 0)$ to $(n, 0)$ is
  $\binom{n-2m}{\ell-m}\binom{2m}{m}$. Summing over all possible $m$ and all
  possible $\ell$ and dividing by $(n+1)$ for the equidistribution of the
  starting point yields \eqref{eq:id:quicksort}.
\end{proof}

Beside this combinatorial proof, we intend to show Theorem~\ref{thm:identity}
in alternative ways. First, we prove that the identity between
\eqref{eq:id:quicksort} and \eqref{eq:id:single}.

\begin{proof}[of \eqref{eq:id:quicksort} equals \eqref{eq:id:single}]
  Consider
  \begin{align*}
    \frac{1}{n+1}\sum_{\ell=m}^{n-m}\frac{\binom{2m}{m}\binom{n-2m}{\ell-m}}{\binom{n}{\ell}}&=
    \frac{1}{n+1}\sum_{\ell=m}^{n-m}\frac{(2m)!\, (n-2m)!\, \ell!\, (n-\ell)!}{m!\, m!\, (\ell-m)!\, (n-\ell-m)!\, n!}\\
    &=\frac{1}{(n+1)\binom{n}{2m}}\sum_{\ell=m}^{n-m}\binom{\ell}{m}\binom{n-\ell}{m}\\
    &=\frac{1}{(n+1)\binom{n}{2m}} \binom{n+1}{2m+1}=\frac{1}{2m+1},
  \end{align*}
  where \cite[(5.26)]{Graham-Knuth-Patashnik:1994} has been used in the
  penultimate step with $\ell=n$, $q=0$, $k=\ell$, $m=m$ and $n=m$. Summing over
  $m$ yields $\sum_{m=0}^{\floor{n/2}} 1/(2m+1) = \Hodd_{n+1}$, thus
  the equality between \eqref{eq:id:single} and \eqref{eq:id:quicksort}.
\end{proof}

It remains to give a computational proof of the equality
between~\eqref{eq:id:double} and \eqref{eq:id:single}. We provide two proofs:
one motivated by ``creative telescoping'' (Appendix~\ref*{sec:creative-telescoping}) and one completely
elementary ``human'' proof (Appendix~\ref*{sec:human-proof}) using not more than Vandermonde's convolution.

\subsection{Proof of the Identity Using Creative Telescoping}
\label{sec:creative-telescoping}
 A computational proof of the identity between \eqref{eq:id:double} and \eqref{eq:id:single} can be generated
by the
summation package Sigma~\cite{Schneider:2015:Sigma-1.81} (see also
Schneider~\cite{Schneider:2007:symb-sum}) together with the packages
HarmonicSums~\cite{Ablinger:2015:HarmonicSums-1.0} and
EvaluateMultiSums~\cite{Schneider:2015:EvaluateMultiSums-0.96}\footnote{The authors thank Carsten Schneider for providing the packages
  Sigma~\cite{Schneider:2015:Sigma-1.81} and
  EvaluateMultiSums~\cite{Schneider:2015:EvaluateMultiSums-0.96}, and
  for his support.}.
To succeed, we have to split the case of even and odd $n$. The obtained proof
certificates are rather lengthy to verify.

Motivated by the previous observations, we also give a proof without additional
machinery. Anyhow, the key step is, as with Sigma, creative telescoping. We
prove an (easier) identity which Sigma comes up with in the following lemma.

\begin{lemma}
  Let
  \begin{align*}
    F(n, \ell)&=\sum_{0\le k<\ell}\frac{\binom{n}{k}}{\binom{n}{\ell}},\\
    G(n, \ell)&=(\ell-1)+ (\ell-1-n)F(n, \ell)
  \end{align*}
  for $0\le \ell\le n$.
  Then
  \begin{equation}\label{eq:CT:simple}
    (n+1)F(n+1, \ell)-(n+2)F(n, \ell)=G(n, \ell+1)- G(n, \ell)
  \end{equation}
  holds for all $0\le \ell< n$.
\end{lemma}

\begin{proof}
  For $0\le \ell<n$, we first compute
  \begin{align}
    F(n+1, \ell)&=\frac{1}{\binom{n+1}{\ell}}\sum_{0\le k<\ell}\binom{n+1}{k}
    =\frac{1}{\binom{n+1}{\ell}}\sum_{0\le k<\ell}\biggl(\binom{n}{k-1}+\binom{n}{k}\biggr)\notag\\
    &=-\frac{\binom{n}{\ell-1}}{\binom{n+1}{\ell}}+\frac{2\binom{n}{\ell}}{\binom{n+1}{\ell}}F(n,
    \ell)
    =-\frac{\ell}{n+1}+2\frac{n+1-\ell}{n+1}F(n,
    \ell)\label{eq:CT:simple:10}
  \end{align}
  and
  \begin{align}
    F(n, \ell+1)&=\frac1{\binom{n}{\ell+1}}\sum_{0\le k<\ell+1}\binom{n}{k}
    =\frac{\binom{n}{\ell}}{\binom{n}{\ell+1}}+
    \frac{\binom{n}{\ell}}{\binom{n}{\ell+1}}F(n, \ell)\notag\\
    &=\frac{\ell+1}{n-\ell} + \frac{\ell+1}{n-\ell}F(n, \ell).\label{eq:CT:simple:01}
  \end{align}
  By plugging \eqref{eq:CT:simple:10} and \eqref{eq:CT:simple:01} into
  \eqref{eq:CT:simple}, all occurrences of $F(n, \ell)$ cancel as well as all
  other terms, which proves \eqref{eq:CT:simple}.
\end{proof}

We are now able to prove the essential recurrence for the sum \eqref{eq:id:double} in
Theorem~\ref{thm:identity}.

\begin{lemma}\label{lemma:CT:recursion-E}
  Let
  \begin{equation*}
    E_n = \frac{4}{n+1}
    \sum_{0\leq k < \ell < \ceil{n/2}} \frac{\binom{n}{k}}{\binom{n}{\ell}}
    + \iverson*{$n$ even} \frac{1}{n+1}
    \left(\frac{2^n}{\binom{n}{n/2}} - 1\right).
  \end{equation*}
  Then
  \begin{equation}\label{eq:CT:recursion-E}
    E_{2N+1}-E_{2N}=0
    \qquad\text{and}\qquad
    E_{2N+2}-E_{2N+1}=\frac1{2N+3}
  \end{equation}
  for $N\ge 0$.
\end{lemma}

\begin{proof}
Multiplying \eqref{eq:CT:simple} with $4/((n+1)(n+2))$ and summing over $0\le
\ell < \ceil{n/2}$ yields
\begin{multline*}
  \frac{4}{n+2}\sum_{0\le k<\ell<
    \ceil{n/2}}\frac{\binom{n+1}{k}}{\binom{n+1}{\ell}}
  -\frac{4}{n+1}\sum_{0\le
    k<\ell<\ceil{n/2}}\frac{\binom{n}{k}}{\binom{n}{\ell}}\\
  =\frac{4}{(n+1)(n+2)}\Bigl(G(n, \ceil{n/2})-G(n, 0)\Bigr)
\end{multline*}
for $n\ge 0$. This is equivalent to
\begin{multline}\label{eq:CT:identity:3}
  \frac{4}{n+2}\sum_{0\le k<\ell<
    \ceil{(n+1)/2}}\frac{\binom{n+1}{k}}{\binom{n+1}{\ell}} - \frac{4[n\text{ even}]}{n+2}F(n+1,n/2)
  -\frac{4}{n+1}\sum_{0\le
    k<\ell<\ceil{n/2}}\frac{\binom{n}{k}}{\binom{n}{\ell}}\\
  =\frac{4}{(n+1)(n+2)}\Bigl(\ceil{n/2}-1+(\ceil{n/2}-1-n)F(n, \ceil{n/2})+1\Bigr).
\end{multline}
We rewrite the double sums in terms of $E_n$ and $E_{n+1}$, respectively, and
use \eqref{eq:mu_n_0}. We also replace
$F(n+1, n/2)$ by \eqref{eq:CT:simple:10}. Then \eqref{eq:CT:identity:3}
is equivalent to
\begin{multline}\label{eq:identity-expectations}
  E_{n+1}-\frac{2[n\text{ odd}]}{n+2}F(n+1, (n+1)/2)-E_n\\  +\frac{2n[n\text{ even}]}{(n+1)(n+2)}-\frac{2[n\text{ even}]}{n+1}F(n, n/2)\\
 =\frac{4}{(n+1)(n+2)}\Bigl(\ceil{n/2}-(\floor{n/2}+1)F(n, \ceil{n/2})\Bigr).
\end{multline}
If $n=2N+1$, equation~\eqref{eq:identity-expectations} is equivalent to
\begin{multline*}
  E_{2N+2}-E_{2N+1}-\frac{2}{2N+3}F(2N+2,N+1)\\
  =\frac{2}{2N+3}-\frac{2}{2N+3}F(2N+1, N+1).
\end{multline*}
Using \eqref{eq:CT:simple:10}, this is equivalent to the second recurrence in \eqref{eq:CT:recursion-E}.

If $n=2N$, then \eqref{eq:identity-expectations} is equivalent to the first
recurrence in \eqref{eq:CT:recursion-E}.
\end{proof}

\begin{proof}[(Computational proof of Theorem~\ref{thm:identity})]
  The definition of $E_0$ in Lemma~\ref{lemma:CT:recursion-E} implies that
  $E_0=0$. Thus \eqref{eq:id:double} and \eqref{eq:id:single}
  coincide for $n=0$. This can be extended to all $n\ge 0$ by induction on $n$
  and Lemma~\ref{lemma:CT:recursion-E}.
\end{proof}

\subsection{Proof of the Identity Using Vandermonde's Convolution}
\label{sec:human-proof}

\begingroup\newcommand{\m}{\ceil{\frac n2}}
We now provide a completely elementary ``human'' proof of the identity between
\eqref{eq:id:double} and \eqref{eq:id:single}.

We first prove an identity on binomial coefficients.

\begin{lemma}\label{lemma:binomial-identity} The identity
\begin{equation*}
\sum_{0\le k\le j}\binom {n}{k}
=\sum_{0\le k \le j}2^k\binom{n-k-1}{j-k}
\end{equation*}
holds for all non-negative integers $j<n$.
\end{lemma}
\begin{proof}
We denote the right hand side by $\rho$. The binomial theorem and symmetry of the
binomial coefficient yield
\begin{align*}
\rho &=\sum_{0\le i\le  k \le j}\binom{k}{i}\binom{n-k-1}{n-1-j}.
\intertext{The sum over $k$ can be evaluated by \cite[(5.26)]{Graham-Knuth-Patashnik:1994}, resulting in}
\rho &=\sum_{0\le i \le j}\binom{n}{n+i-j}.
\intertext{Symmetry of the binomial coefficient and then replacing $i$ by $j-i$
lead to}
\rho &=\sum_{0\le i \le j}\binom{n}{j-i}=\sum_{0\le i \le j}\binom{n}{i}.
\end{align*}
\end{proof}

We are now able to 
establish a recurrence satisfied 
by the double sum in \eqref{eq:id:double}.

\begin{lemma}\label{lemma:recursion-binomial}For $n\ge 0$, let
  \begin{equation*}
    S_n=\sum_{0\le k< \ell< \m} \frac{\binom{n}{k}}{\binom{n}{\ell}}.
  \end{equation*}
  Then the recurrence
  \begin{equation*}
    S_n =  \frac{n+1}{n-1}S_{n-2}+\frac{n+1}{4n}
    + \iverson*{$n$ even}\Bigl(\frac{2^{n-2}}{n\binom{n}{n/2}} -
    \frac{1}{2(n-1)} - \frac{1}{4n} \Bigr)
  \end{equation*}
  holds for $n\ge 2$.
\end{lemma}
\begin{proof}
We replace the sum in the numerator of $S_n$ using Lemma~\ref{lemma:binomial-identity}
and obtain
\begin{align}
S_n&=\sum_{0\le k\le \ell< \m} \frac{\binom{n}{k}}{\binom{n}{\ell}}-\m\notag\\
&=\sum_{0\le k\le \ell< \m} \frac{2^k\binom{n-k-1}{\ell-k}}{\binom{n}{\ell}}-\m\notag\\
&=\sum_{0\le k\le \ell< \m} \frac{2^k(n-k-1)!\,k!}{n!}\binom{\ell}{k}\bigl((n-k)-(\ell-k)\bigr)-\m\notag\\
&=\sum_{0\le k\le \ell< \m} \frac{2^k(n-k)!\,k!}{n!}\binom{\ell}{k}
\notag\\&\qquad-\sum_{0\le k\le \ell< \m} \frac{2^k(n-k-1)!\,(k+1)!}{n!}\binom{\ell}{k+1}-\m.\notag
\intertext{The sum over $\ell$ can be evaluated using upper summation, see \cite[Table~174]{Graham-Knuth-Patashnik:1994}:}
S_n&=\sum_{0\le k< \m} \frac{2^k(n-k)!\,k!}{n!}\binom{\m}{k+1}
\notag\\&\qquad-\sum_{0\le k< \m} \frac{2^k(n-k-1)!\,(k+1)!}{n!}\binom{\m}{k+2}-\m.\notag
\intertext{Shifting the summation index in the second sum leads to}
S_n&=\sum_{0\le k< \m} \frac{2^k(n-k)!k!}{n!}\binom{\m}{k+1}
\notag\\&\qquad-\sum_{0\le k< \m} \frac{2^{k-1}(n-k)!k!}{n!}\binom{\m}{k+1}-\m+
\frac{1}{2}\m\notag\\
&=\frac{1}{2n!}\sum_{0\le k< \m} 2^k(n-k)!k!\binom{\m}{k+1}-\frac12\m\notag\\
&=\frac{\m!}{2n!}\sum_{0\le k< \m}
2^k\frac{(n-k)!}{(\m-k-1)!}\frac1{k+1}-\frac12\m\label{S_n_explicit}.
\end{align}
We intend to derive a recurrence linking $S_n$ with
$S_{n-2}$. Therefore, for $n\ge 2$, we rewrite $S_n$ as
\begin{align*}
  S_n&=\frac{\m!}{2n!}\sum_{0\le k< \m-1}
  2^k\frac{(n-k-2)!}{(\m-k-2)!}\frac{(n-k)(n-k-1)}{(\m-k-1)(k+1)} - \frac{\m}{2} +
  \frac{2^{\m-2}}{\binom{n}{\m-1}}.
\intertext{Partial fraction decomposition in $k$ yields}
  S_n&=\frac{\m!}{2n!}\sum_{0\le k< \m-1}
  2^k\frac{(n-k-2)!}{(\m-k-2)!}\Bigl( - 1+\frac{(n-\m)(n-\m+1)}{{(\m-k-1)} \m} + \frac{n(n+1)}{{(k +
        1)} \m}\Bigr)\\
  &\qquad- \frac{\m}{2} + \frac{2^{\m-2}}{\binom{n}{\m-1}}\\
  &=-\frac{\m!(n-\m)!}{2n!}\sum_{0\le k< \m-1}
  2^k\binom{n-k-2}{\m-k-2}\\
  &\qquad +\frac{(\m-1)!(n-\m+1)!}{2n!}\sum_{0\le k<
    \m-1}2^k\binom{n-k-2}{\m-k-1}\\
  &\qquad +\frac{(\m-1)!(n+1)n}{2n!}\sum_{0\le k< \m-1}
  2^k\frac{(n-k-2)!}{(\m-k-2)!}\frac{1}{{(k +
        1)}}\\
  &\qquad- \frac{\m}{2} + \frac{2^{\m-2}}{\binom{n}{\m-1}}.
\end{align*}
We again use Lemma~\ref{lemma:binomial-identity} for the first two summands and
\eqref{S_n_explicit} with $n$ replaced by $n-2$ for the third summand
to
obtain
\begin{align*}
  S_n&=-\frac{1}{2\binom{n}{\m}}\sum_{0\le k< \m-1}\binom{n-1}{k}
  +\frac{1}{2\binom{n}{\m-1}}\sum_{0\le k<
    \m}\binom{n-1}{k}\\
  &\qquad +\frac{n+1}{n-1}\Bigl(S_{n-2}+\frac{\m-1}2\Bigr)- \frac{\m}{2}.
\end{align*}
Adding another summand for $k=\m-1$ to the first sum leads to
\begin{align*}
  S_n&=\biggl(\frac{1}{2\binom{n}{\m-1}}-\frac{1}{2\binom{n}{\m}}\biggr)\sum_{0\le
    k< \m}\binom{n-1}{k}\\
   &\qquad
   +\frac{\binom{n-1}{\m-1}}{2\binom{n}{\m}}+\frac{n+1}{n-1}S_{n-2}+\frac{\m-1}{n-1}
   -\frac12\\
  &=\frac{n-2\m+1}{2\m\binom{n}{\m}}\sum_{0\le
    k< \m}\binom{n-1}{k}\\
   &\qquad +\frac{n+1}{n-1}S_{n-2}+\frac{\m}{2n}+\frac{\m-1}{n-1} -\frac12.
\end{align*}
If $n$ is odd, then $n-2\m+1=0$ so that
the first summand vanishes. The result follows in that case.

For even $n$, the remaining sum is $2^{n-2}$ and the result follows.
\end{proof}

\begin{proof}[(Computational proof of Theorem~\ref{thm:identity})]
  Denote the expression in~\eqref{eq:id:double} by $E_n$. From
  Lemma~\ref{lemma:recursion-binomial}, we obtain the recurrence
  \begin{equation*}
    E_n = E_{n-2}+\frac{1}{n+\iverson*{$n$ even}}
  \end{equation*}
  for $n\ge 2$. As $E_0=E_1=1$, this implies that $E_n=\Hodd_{n+1}$. Thus
  \eqref{eq:id:double} and \eqref{eq:id:single} coincide.
\end{proof}

\endgroup
}


\def\appendixasy{
\section{Asymptotic Aspects}
\label{sec:asymptotics}
\begin{lemma}\label{lem:harmonic-asy}
  We have
  \begin{subequations}
  \begin{align}
    H_n &= \log n + \gamma + \frac{1}{2n} - \frac{1}{12n^2}
    + \Oh[Big]{\frac{1}{n^4}},\label{eq:asymptotic-harmonic-number}\\
    \Hodd_n &= \frac12 \log n
    + \frac{\gamma + \log 2}{2} + \frac{\iverson*{$n$ odd}}{2n}
    + \frac{3\iverson*{$n$ even}-2}{12n^2}
    + \Oh[Big]{\frac{1}{n^4}},\label{eq:asymptotic-odd-harmonic-number}\\
    \Halt_n &= -\log 2 + \frac{(-1)^n}{2n}
    - \frac{(-1)^n}{4n^2}
    + \Oh[Big]{\frac{1}{n^4}}\label{eq:asymptotic-alternating-harmonic-number}
  \end{align}
  \end{subequations}
  as $n$ tends to $\infty$, and where $\gamma = 0.5772156649\!\dots$ is the
  Euler--Mascheroni constant.
\end{lemma}

\begin{proof}
  The asymptotic expansion~\eqref{eq:asymptotic-harmonic-number} is well-known,
  cf.\@ for instance \cite[(9.88)]{Graham-Knuth-Patashnik:1994}.

  We write $\alpha_n=\iverson*{$n$ even}$ and thus
  $\floor{n/2}=(n-1+\alpha_n)/2$. As $\alpha_n$ is obviously bounded, we can
  simply plugin this expression into the asymptotic expansions and simplify all
  occurring higher powers of $\alpha_n$ by the fact that $\alpha_n^2=\alpha_n$. Using the relations
  $\Hodd_n=H_n- H_{\floor{n/2}}/2$ and $\Halt_n=H_n - 2\Hodd_n$ leads to the
  expansions~\eqref{eq:asymptotic-odd-harmonic-number} and
  \eqref{eq:asymptotic-alternating-harmonic-number}, respectively.

  The actual asymptotic computations\footnote{A worksheet containing
    the computations can be found at
    \url{http://www.danielkrenn.at/downloads/quicksort-paths/quicksort-paths.ipynb}.}
  have been carried out using the asymptotic
  expansions module \cite{Hackl-Krenn:2015:asy-sage} of
  SageMath~\cite{SageMath:2016:7.0}.
\end{proof}

\begin{corollary}\label{cor:main-asy}
  The expected number of zeros is
  \begin{multline*}
    \E{Z_n} =
    \frac12 \log n + \frac{\gamma+\log2}{2} \\
    + \frac{1+\iverson*{$n$ even}}{2n}
    - \frac{2+9\iverson*{$n$ even}}{12n^2}
    + \frac{\iverson*{$n$ even}}{n^3}
    + \Oh[Big]{\frac{1}{n^4}}
  \end{multline*}
  asymptotically as $n$ tends to infinity.
\end{corollary}

\begin{proof}
  Combine Theorem~\ref{thm:paths-prob} and Lemma~\ref{lem:harmonic-asy}.
\end{proof}
}


\def\appendixdistribution{
\section{Distribution}
\label{sec:distribution}

In this part of this article, we study the distribution of the number
of zeros. As for the expectation $\E{Z_n}$, we get again an exact
formula, as well as an asymptotic formula. We start with the former, which is
stated in the following theorem; the latter is stated directly afterwards.

\begin{theorem}\label{thm:distribution}
  Let $r\in\N_0$. For positive lengths~$n\geq 2r-2$, the probability that a
  randomly chosen path~$P_n$ has exactly~$r$ zeros is
  \begin{align*}
    \P{Z_n = r} &= \frac{2^r}{n+1}
    \frac{\binom{\ceil{n/2}}{r}}{\binom{n}{r}}
    \bigg( \frac{2\ceil{n/2}}{r(r+1)} + \frac{r-1}{r+1} +
        \iverson*{$n$ even} \frac{1}{r} \bigg) \\
    &\phantom{=}\;+ \iverson*{$n$ even}
    \frac{2^{r-1}(r-1)}{(n+1)r} \frac{\binom{n/2}{r-1}}{\binom{n}{r}}
  \end{align*}
  and we have $\P{Z_0 = r} = \iverson{r=1}$. Otherwise $\P{Z_n = r} = 0$.
\end{theorem}

This exact formula admits a local limit theorem towards a
discrete distribution. The details are as follows.

\begin{corollary}\label{cor:distribution-asy}
  Let $0<\eps\leq\frac12$. For positive integers $r$ with $r =
  \Oh{n^{1/2-\eps}}$, we have asymptotically
  \begin{equation*}
    \P{Z_n = r} = \frac{1}{r(r+1)} \left(1 + \Oh{1/n^{2\eps}}\right)
  \end{equation*}
  as $n$ tends to infinity.
\end{corollary}

\begin{proof}[of Theorem~\ref{thm:distribution}]
  Again, we assume $s\geq0$ (by symmetry of the generating
  function~\eqref{eq:path-gf-zeros}). Note that $Q_s$ counts the
  number of zeros by the variable~$u$ except for the last zero (at
  $(n,0)$). By starting with Lemma~\ref{lem:transform-gf-zeros} and
  some rewriting, we can extract the $(r-1)$st coefficient with
  respect to $u$ as
  \begin{equation*}
    [u^{r-1}] Q_s(z,u)
    = [u^{r-1}] \frac{v^s}{1-u\frac{2v^2}{1+v^2}}
    = \frac{2^{r-1} v^{2(r-1)+s}}{(1+v^2)^{r-1}}.
  \end{equation*}
  Next, we extract the coefficient of $z^n$. We use
  Lemma~\ref{lem:extract-coeffs-worlds} to obtain
  \begin{align*}
    [z^n u^{r-1}] Q_s(z,u)
    &= [v^n] (1-v^2) (1+v^2)^{n-1} \frac{2^{r-1} v^{2(r-1)+s}}{(1+v^2)^{r-1}}\\
    &=2^{r-1}\,[v^{n-s-2(r-1)}] (1-v^2) (1+v^2)^{n-r}.
  \end{align*}
  We set $\ell=\frac12(n-s)$ and get
  \begin{align*}
    [z^n u^{r-1}] Q_s(z,u)
    &= 2^{r-1}\,[v^{2\ell-2r+2}] (1-v^2) (1+v^2)^{n-r} \\
    &= 2^{r-1}\,[v^{\ell-r+1}] (1-v) (1+v)^{n-r} \\
    &= 2^{r-1} \binom{n-r}{\ell-r+1} - 2^{r-1} \binom{n-r}{\ell-r} \\
    &= 2^{r-1} \binom{n-r}{n-\ell-1} - 2^{r-1} \binom{n-r}{n-\ell}.
  \end{align*}
  Note that we have to assume $n-r\geq0$ to make this work. Otherwise,
  anyhow, there are no paths with exactly $r$ zeros (and positive
  length~$n$).

  If $\ell>r-1$ we can rewrite the previous formula to obtain
  \begin{equation*}
    [z^n u^{r-1}] Q_s(z,u)
    = 2^{r-1} \binom{n-r}{n-\ell} \left(\frac{n-\ell}{\ell-r+1} - 1\right)
    = 2^{r-1} \frac{s+r-1}{\ell-r+1} \binom{n-r}{n-\ell},
  \end{equation*}
  if $\ell=r-1$, then we have $[z^n u^{r-1}] Q_s(z,u) = 2^{r-1}$ (independently of~$n$),
  and if $\ell<r-1$ we get $[z^n u^{r-1}] Q_s(z,u) = 0$.

  To finish the proof, we have to normalize this number of paths with exactly $r$
  zeros and then sum up over all $\ell$. So let us start with the
  normalization part. We set
  \begin{equation*}
    \lambda_{n,r,s} = \P{Z_n = r \mid
      \text{$P_n$ starts in $(0,s)$}}
    = \frac{[z^n u^{r-1}] Q_s(z,u)}{[z^n] Q_s(z,1)}
  \end{equation*}
  for $n\equiv s\pmod 2$ and $\lambda_{n,r,s} = 0$ otherwise. The denominator
  $[z^n] Q_s(z,1)$ was already determined in Lemma~\ref{lem:coeffs-zeros}.

  If $\ell>r-1$, we have
  \begin{align*}
    \lambda_{n,r,s}
    &= 2^{r-1} \frac{s+r-1}{\ell-r+1} \binom{n-r}{n-\ell} \Big/ \binom{n}{\ell} \\
    &= 2^{r-1} \frac{s+r-1}{\ell-r+1} \frac{(n-r)!\,\ell!\,(n-\ell)!}{
      (n-\ell)!\,(\ell-r)!\,n!} \\
    &= 2^{r-1} \frac{(n-r)!}{n!} \frac{\ell!\,(s+r-1)}{(\ell-r+1)!},
  \end{align*}
  where the last line magically holds for $\ell=r-1$ as well. In particular, we
  obtain
  \begin{equation*}
    \lambda_{n,r,0}
    = 2^{r-1} (r-1) \iverson{n \geq 2r-2} \frac{(n/2)!}{n!} \frac{(n-r)!}{(n/2-r+1)!}
    = \frac{2^{r-1} (r-1)}{r} \binom{n/2}{r-1} \Big/ \binom{n}{r}.
  \end{equation*}

  We have arrived at the summation of the $\lambda_{n,r,s}$. The result
  follows as
  \begin{align*}
    \P{Z_n = r}
    &= \sum_{s=-n}^n \P{Z_n = r \mid \text{$P_n$ starts in $(0,s)$}}
    \P{\text{$P_n$ starts in $(0,s)$}} \\
    &= \frac{1}{n+1} \sum_{s=-n}^n \lambda_{n,r,s} \\
    &= \frac{2}{n+1} \sum_{\ell=0}^{\ceil{n/2}-1} \lambda_{n,r,n-2\ell}
    + \frac{1}{n+1} \lambda_{n,r,0},
  \end{align*}
  and plugging in $\lambda_{n,r,s}$ gives the intermediate result
  \begin{equation}\label{eq:distribution-intermediate}
  \begin{split}
    \P{Z_n = r}
    &= 2^{r} \frac{(n-r)!}{(n+1)!} \sum_{\ell=r-1}^{\ceil{n/2}-1}
     \frac{\ell!\,(n-2\ell+r-1)}{(\ell-r+1)!} \\
    &\phantom{=}\;+ \iverson*{$n$ even}
    \frac{2^{r-1}(r-1)}{(n+1)r} \binom{n/2}{r-1} \Big/ \binom{n}{r}
  \end{split}
  \end{equation}
  for $n\geq r$. At this point, we interrupt this proof to evaluate the sum
  over the~$\ell$ and continue afterwards.
\end{proof}

\begin{lemma}\label{lem:distribtion-sum}
  We have
  \begin{equation*}
     \sum_{\ell=r-1}^{\ceil{n/2}-1}
     \frac{\ell!\,(n-2\ell+r-1)}{(\ell-r+1)!}
     = r!\, \binom{\ceil{n/2}}{r}
     \bigg( \frac{2\ceil{\frac{n}{2}}}{r(r+1)} + \frac{r-1}{r+1} +
       \iverson*{$n$ even} \frac{1}{r} \bigg).
  \end{equation*}
\end{lemma}

Of course, this lemma can be proved computationally by, for example,
Sigma~\cite{Schneider:2015:Sigma-1.81}. However, we give a direct proof here.

\begin{proof}
  We obtain
  \begin{align*}
     \sum_{\ell=r-1}^{\ceil{n/2}-1}
     \frac{\ell!\,(n-2\ell+r-1)}{(\ell-r+1)!} &=
     (r-1)! \sum_{\ell=r-1}^{\ceil{n/2}-1}
     \binom{\ell}{r-1}(n-2\ell+r-1)\\
     &=(r-1)! \sum_{\ell=r-1}^{\ceil{n/2}-1} \binom{\ell}{r-1}(n+1+r-2(\ell+1))\\
     &=(n+1+r) (r-1)! \sum_{\ell=r-1}^{\ceil{n/2}-1} \binom{\ell}{r-1}-2r!
     \sum_{\ell=r-1}^{\ceil{n/2}-1} \binom{\ell+1}{r}.
     \intertext{Using upper summation,
       cf.~\cite[5.10]{Graham-Knuth-Patashnik:1994}, yields}
     \sum_{\ell=r-1}^{\ceil{n/2}-1}
     \frac{\ell!\,(n-2\ell+r-1)}{(\ell-r+1)!}
     &=(n+1+r)(r-1)! \binom{\ceil{n/2}}{r} - 2r!\binom{\ceil{n/2}+1}{r+1}\\
     &=(r-1)!\binom{\ceil{n/2}}{r}\left(n+1+r - 2\frac{r}{r+1}(\ceil{n/2}+1)\right)\\
     &=(r-1)!\binom{\ceil{n/2}}{r}\left(n+r-1 -
       2\ceil{n/2}+\frac{2\ceil{n/2}}{r+1}+\frac{2}{r+1}\right).
  \end{align*}
  The lemma follows by replacing $n-2\ceil{n/2}$ with $\iverson*{$n$
      even} - 1$ and by collecting terms.
\end{proof}

\begin{proof}[of Theorem~\ref{thm:distribution} continued]
  To finish the proof, we simply plug in the result of
  Lemma~\ref{lem:distribtion-sum} into~\eqref{eq:distribution-intermediate} and
  rewrite the factorials as binomial coefficients.
\end{proof}

As a next step, we want to prove Corollary~\ref{cor:distribution-asy}, which
extracts the asymptotic behaviour of the distribution
(Theorem~\ref{thm:distribution}). To show that asymptotic formula, we
will use
the following auxiliary result.

\begin{lemma}\label{lem:quo-factorials}
 Let $0<\eps\leq\frac12$. For integers $c$ with $c =
 \Oh{N^{1/2-\eps}}$ we have
 \begin{equation*}
   c!\,\binom{N}{c}  = N^c(1+\Oh{1/N^{2\eps}}).
 \end{equation*}
\end{lemma}

\begin{proof}
The inequality $N^c \ge c!\,\binom{N}{c}$ is trivial. We observe
 \begin{align*}
   c!\,\binom{N}{c}
   &= N^{c}\cdot\prod_{0\le i < c}\left(1-\frac{i}{N}\right)
   \ge N^{c}\cdot\biggl(1-\sum_{0\le i < c}\frac{i}{N}\biggr)
   \ge N^c \left(1-\frac{c^2}{2N}\right)\\
   &= N^c \left(1+\Oh{1/N^{2\eps}}\right),
 \end{align*}
where the assumption on $c$ has been used in the last step.
\end{proof}

\begin{proof}[of Corollary~\ref{cor:distribution-asy}]
  By using Lemma~\ref{lem:quo-factorials}, the exact result of
  Theorem~\ref{thm:distribution} becomes
  \begin{align*}
    \P{Z_n = r}
    &= \frac{2^r}{n} \frac{n^r}{2^r} \frac{1}{n^r}
    \left(\frac{n}{r(r+1)} + \Oh{1}\right)
    \left(1+\Oh{1/n^{2\eps}}\right) \\
    &\phantom{=}\;+ \iverson*{$n$ even}
    \frac{2^{r-1}(r-1)}{n} \frac{n^{r-1}}{2^{r-1}} \frac{1}{n^r}
    \left(1+\Oh{1/n^{2\eps}}\right) \\
    &= \frac{1}{r(r+1)} \left(1 + \Oh{1/n^{2\eps}}\right)
  \end{align*}
  as claimed.
\end{proof}
}


\part{Part II: More Lattice Paths and Zeros}
\label{sec:more-lattice-paths}

This second part deals with an analysis of some special
zero-configurations, which are needed for the analysis of the
partitioning strategies in Part~\ref{sec:quicksort}. Moreover, in
Section~\ref{sec:lattice-paths-N}, we extend the model introduced in
Section~\ref{sec:description} to accommodate lattice paths of
variable length. Again expectations are studied exactly and asymptotically.


\section{Going to Zero and Coming From Zero}
\label{sec:more-zeros}
For the analysis of comparison-optimal dual-pivot quicksort algorithms
(see Part~\ref{sec:quicksort}) we need the following
two variants of zeros on the lattice path.
\begin{itemize}
\item An \emph{up-to-zero situation} is a point $(x, 0) \in P_n$ such
  that $(x - 1, - 1) \in P_n$.
\item A \emph{down-from-zero situation} is a
  point $(x, 0) \in P_n$ such that $(x + 1, -1) \in P_n$.
\end{itemize}
\def\appendixtofromzerodef{
For a path $P_n$ of length $n$
chosen according to the probabilistic model from
Section~\ref{sec:description} we define the random variables
\begin{align*}
  Z^{\nearrow}_n &= \text{number of up-to-zero situations on $P_n$}
  \intertext{and}
  Z^{\searrow}_n &= \text{number of down-from-zero situations on $P_n$}.
\end{align*}
}
We show
  \begin{equation*}
    \E{\text{number of up-to-zero situations on $P_n$}} =
    \frac12\Bigl(\E{Z_n} - \frac{\iverson*{$n$ even}}{n+1}\Bigr)
    = \frac12 \Hodd_n
  \end{equation*}
    and
  \begin{equation*}
    \E{\text{number of down-from-zero situations on $P_n$}} =
    \frac12\left(\E{Z_n} - 1\right)
    = \frac12 \bigl(\Hodd_{n+1} - 1\bigr).
  \end{equation*}

\begin{proof}[idea]
The factor $\frac12$ stems from symmetry: Up-to-zero situations at $(x, 0)$
occur with the same probability as the symmetric ``down-to-zero'' situations at $(x,0)$, 
similarly for down-from-zero situations.  
The correction terms $\frac{\iverson*{$n$ even}}{n+1}$ and $1$ are caused by the
fact that there is a zero, but no up-to-zero situation, at $(0,0)$, and a zero, 
but no down-from-zero situation, at $(n,0)$. The full proofs are in 
Appendix~\ref*{sec:appendix:more-zeros}.
\end{proof}

\def\appendixtofromzero{
\begin{lemma}\label{lem:paths-prob:updown}
  For a randomly (as described in Section~\ref{sec:description}) chosen path of
  length~$n$, 
  \begin{align*}
    \E{Z^\nearrow_n} &=
    \frac12\left(\E{Z_n} - \frac{\iverson*{$n$ even}}{n+1}\right)
    = \frac12 \Hodd_n
    \intertext{and}
    \E{Z^\searrow_n} &= \frac12\left(\E{Z_n} - 1\right)
    = \frac12 \bigl(\Hodd_{n+1} - 1\bigr).
  \end{align*}
\end{lemma}

\begin{proof}[of Lemma~\ref{lem:paths-prob:updown}]
In the proof of Theorem~\ref{thm:paths-prob} linearity of the expectation gave that
\begin{equation}
    \E{Z_n} = \sum_{m=0}^{n} \P{\text{path~$P_n$ runs through $(n-m,0)$}}.
    \label{eq:proof:paths-prob:zeroes}
\end{equation}
Similarly we have
\begin{equation*}
    \E{Z^\nearrow_n} = \sum_{m=0}^{n-1} \P{\text{path~$P_n$ runs through $(n-m-1,-1)$ and $(n-m,0)$}}.
\end{equation*}
Note that the sum for $\E{Z_n}$ has a term for $m=n$ but the one for $\E{Z^\nearrow_n}$ does not.
Moreover, for $m<n$ the term
\begin{equation*}
  \P{\text{path~$P_n$ runs through $(n-m-1,-1)$ and $(n-m,0)$}}
\end{equation*}
is exactly half of
$\P{\text{path~$P_n$ runs through $(n-m,0)$}}$, since for every path through $(n-m-1,-1)$ and $(n-m,0)$
there is one with the same probability through $(n-m-1,1)$ and $(n-m,0)$ by
symmetry (or use the transition probabilities from \eqref{eq:transition-probabilities}). 
So
\begin{equation*}
  \E{Z^\nearrow_n}
  = \frac12 \left(\E{Z_n} - \frac{\iverson*{$n$ even}}{n+1}\right)
  = \frac12 \sum_{m=1}^{n} \frac{\iverson*{$m$ odd}}{m}
  = \frac12 \Hodd_n.
 \end{equation*}

Similarly, comparing \eqref{eq:proof:paths-prob:zeroes} with
\begin{equation*}
    \E{Z^\searrow_n} = \sum_{m=1}^{n} \P{\text{path~$P_n$ runs through $(n-m,0)$ and $(n-m+1,-1)$}}
\end{equation*}
we see that $\E{Z_n}$ has a term for $m=0$ but $\E{Z^\searrow_n}$ does not. For
$m>0$ exactly half of the paths that run through $(n-m,0)$ have a down-from-zero situation at $(n-m,0)$.
So
\begin{equation*}
  \E{Z^\searrow_n}
  = \frac12(\E{Z_n} - 1)
  = \frac12 \bigl(\Hodd_{n+1} - 1\bigr).
\end{equation*}
\end{proof}

Applying Corollary~\ref{cor:main-asy}, we get the following asymptotic expansion.

\begin{corollary}
  The expected number of up-to-zero situations is
  \begin{align*}
    \E{Z^\nearrow_n} &=
    \frac14 \log n + \frac{\gamma+\log2}{4}
    + \frac{\iverson*{$n$ odd}}{4n}
    + \frac{3\iverson*{$n$ even} - 2}{24n^2}
    + \Oh[Big]{\frac{1}{n^4}},
    \intertext{and the expected number of down-from-zero situations is}
    \E{Z^\searrow_n} &=
    \frac14 \log n + \frac{\gamma+\log2-2}{4} \\
    &\phantom{=}\hphantom{0}
    + \frac{\iverson*{$n$ even} + 1}{4n}
    - \frac{9\iverson*{$n$ even} + 2}{24n^2}
    + \frac{\iverson*{$n$ even}}{2n^3}
    + \Oh[Big]{\frac{1}{n^4}},
  \end{align*}
  asymptotically as $n$ tends to infinity.
  \label{cor:up:down:asym} 
\end{corollary}
}


\section{Lattice Paths of Variable Length}
\label{sec:lattice-paths-N}


In this section, we use a random variable~$N'$ instead of the fixed~$n$
above. Let us fix an~$n\in\N$ with $n\geq2$. We choose a path
length~$N'$ according to the following rules.
\begin{enumerate}
\item Choose $(P,Q)$ with $1\le P < Q \le n$
 uniformly at random from all $\binom{n}{2}$ possibilities.
\item Let $N' = n - 1 - (Q-P)$.
\item Choose a path of length $N'$ according to
Section~\ref{sec:description}.
\end{enumerate}
Let us denote the number of up-to-zero and down-from-zero situations
on the path by $X^\nearrow_{n}$ and $X^\searrow_{n}$, respectively.
In Appendix~\ref*{sec:appendix:lattice-paths-N}, we show
  \begin{align*}
    \E{X^\nearrow_{n}} &= \frac{1}{2\binom{n}{2}}
    \sum_{n'=0}^{n-2} \sum_{m=1}^{n'} \iverson*{$m$ odd}\frac{n'+1}{m} =
    \frac12 \Hodd_{n-2} - \frac18 +\frac{(-1)^n}{8(n-\iverson*{$n$
    even})}
    \intertext{and}
    \E{X^\searrow_{n}} &= \frac{1}{2\binom{n}{2}}
    \sum_{n'=0}^{n-2} \sum_{m=3}^{n'+1} \iverson*{$m$ odd}\frac{n'+1}{m}
    =\E{X^\nearrow_{n}} - \frac12 +\frac1{2(n-\iverson*{$n$ even})}.
  \end{align*}

\def\appendixvarlen{
The following proposition will be proven in this section.

\begin{proposition}\label{pro:dual-pivot-expect-exact}
  For a randomly chosen path (as described above), the expected number
  of up-to-zero situations is
  \begin{align}
    \E{X^\nearrow_{n}} &= \frac{1}{2\binom{n}{2}}
    \sum_{n'=0}^{n-2} \sum_{m=1}^{n'} \iverson*{$m$ odd}\frac{n'+1}{m} =
    \frac12 \Hodd_{n-2} - \frac18 +\frac{(-1)^n}{8(n-\iverson*{$n$ even})}
    \label{eq:E_X_up_n_exact}
    \intertext{and the expected number of down-from-zero situations is}
    \E{X^\searrow_{n}} &= \frac{1}{2\binom{n}{2}}
    \sum_{n'=0}^{n-2} \sum_{m=3}^{n'+1} \iverson*{$m$ odd}\frac{n'+1}{m}
    =\E{X^\nearrow_{n}} - \frac12 +\frac1{2(n-\iverson*{$n$ even})}.
        \label{eq:E_X_down_n_exact}
  \end{align}
  The corresponding generating functions are
  \begin{align*}
    \sum_{n\ge 2}\E{X^\nearrow_{n}} z^{n} &=\frac{\artanh(z)}{2  {(1-z)}}-
    \frac{z^2}{8  {(1-z)}} - \frac{3z + 5}{8}  \artanh(z) + \frac{1}{8}  z ,\\
    \sum_{n\ge 2}\E{X^\searrow_{n}} z^{n} &=\frac{\artanh(z)}{2  {(1-z)}}
  - \frac{5z^2}{8  {(1-z)}} + \frac{z - 1}{8}  \artanh(z) - \frac{3}{8}  z.
  \end{align*}
  Asymptotically, we have
  \begin{align*}
    \E{X^\nearrow_{n}} &= \frac14 \log n
    + \frac{2\gamma + 2\log2 - 1}{8}
    - \frac{3}{8n} - \frac{3\iverson*{$n$ even} + 1}{12n^2} - \frac{3\iverson*{$n$ even}}{8n^3} + \Oh[Big]{\frac{1}{n^4}},\\
   \E{X^\searrow_{n}} &= \frac14 \log n
   + \frac{2\gamma + 2\log2 - 5}{8}
+ \frac{1}{8n} + \frac{3\iverson*{$n$ even} - 1}{12n^2} + \frac{\iverson*{$n$ even}}{8n^3} 
    + \Oh[Big]{\frac{1}{n^4}},
  \end{align*}
  as $n\to\infty$.
\end{proposition}

To prove this proposition, we first compute the distribution of $N'$. The
following lemma is a simple consequence of the definition of~$N'$.

\begin{lemma}\label{lem:dist-N}
  If $n'\in\set{0,1,\dots,n-2}$, then
  \begin{equation*}
    \P{N' = n'} = \frac{n'+1}{\binom{n}{2}},
  \end{equation*}
  and otherwise~$\P{N' = n'} = 0$.
\end{lemma}

\begin{proof}
  Fix $n'\in\set{0,1,\dots,n-2}$, and set $\Delta=n - 1 - n'$. There
  are $n-\Delta=n'+1$ choices of the random variable~$P$ such that
  $Q=P+\Delta$ is still at most~$n$. Thus, the lemma follows.
\end{proof}

\begin{proof}[of Proposition~\ref{pro:dual-pivot-expect-exact}]
  For any random variable $Y_{n}$ depending on $N'$, the law of total expectation yields
  \begin{equation}\label{eq:cond-exp-n0}
    \E{Y_{n}} = \sum_{n'\in\N_0}
    \P{N'=n'} \E[empty]{Y_{n} \mid N'=n'}.
  \end{equation}

  From \eqref{eq:cond-exp-n0} and the
  Lemmata~\ref{lem:paths-prob:updown} and~\ref{lem:dist-N}, we immediately 
	get that $\E{X^\nearrow_{n}}$ and $\E{X^\searrow_{n}}$
	are equal to the respective double sum given in 
	\eqref{eq:E_X_up_n_exact} and \eqref{eq:E_X_down_n_exact}.

  To obtain the difference, one observes that
  \begin{equation*}
    \E[normal]{Z^\nearrow_{n'}}-\E[normal]{Z^\searrow_{n'}}
    =\frac12\Bigl(1-\frac{\iverson*{$n'$ even}}{n'+1}\Bigr),
  \end{equation*}
  from which we obtain
  \begin{equation*}
    \E{X^\nearrow_{n}}-\E{X^\searrow_{n}} = \frac{1}{2\binom{n}{2}}
    \sum_{n'=0}^{n-2} \left(1-\frac{\iverson*{$n'$ even}}{n'+1}\right)\cdot(n'+1).
  \end{equation*}
  Simplifying this to $\frac12-\frac{1}{2(n - [n\text{ even}])}$ is straightforward.

  It remains to verify the second expression for $\E{X^\nearrow_{n}}$ given in \eqref{eq:E_X_up_n_exact}.
  We have   
  \begin{align*}
    \E{X^\nearrow_{n}} 
    &= 
    \frac{1}{2\binom{n}{2}}
    \sum_{n'=0}^{n-2} \sum_{m=1}^{n'} \iverson*{$m$ odd}\frac{n'+1}{m} \\
    &=
    \frac{1}{2\binom{n}{2}}
    \sum_{m=1}^{n-2} \frac{\iverson*{$m$ odd}}{m} \sum_{n'=m}^{n-2} (n'+1) \\
    &=
    \frac{1}{2\binom{n}{2}}
    \sum_{m=1}^{n-2} \iverson*{$m$ odd}\frac{(n-1)n - m(m+1)}{2m}.
  \end{align*}
  (We just summed the arithmetic series.)
  So
  \begin{align*}
    \E{X^\nearrow_{n}} 
    &=
    \frac{1}{2\binom{n}{2}}
    \sum_{m=1}^{n-2} \iverson*{$m$ odd}\left( \frac{\binom{n}{2}}{m} - \frac{m+1}{2}\right)\\
    &=
    \frac{1}{2}\Hodd_{n-2} - \frac{1}{4\binom{n}{2}} \sum_{m=2}^{n-1}
    \iverson*{$m$ even}m\\
    &=
    \frac{1}{2}\Hodd_{n-2} - \frac{\floor{(n-1)/2}\floor{(n+1)/2}}{4\binom{n}{2}}.
  \end{align*}
  This leads to~\eqref{eq:E_X_up_n_exact}.

  We now compute the generating functions.
  We first note that for $k\in\Z$, we have
  \begin{equation}\label{eq:artanh}
    \sum_{n> k}\frac{\iverson*{$n-k$ odd}}{n-k}z^{n}=
    z^k\artanh(z).
  \end{equation}

  Taking the summatory function amounts to division by $(1-z)$ and a shift in
  the argument corresponds to multiplication by $z$, so we obtain
  \begin{equation*}
    \sum_{n\ge 2}\Hodd_{n-2}z^n = z^2\frac{\artanh(z)}{1-z}.
  \end{equation*}

  The
  remaining summands lead to geometric series and another instance of
  \eqref{eq:artanh}, so we obtain
  \begin{align*}
    \sum_{n\ge 2}\E{X^\nearrow_{n}} z^{n} =
    \frac{z^2\artanh(z)}{2(1-z)}-\frac{z^2}{8(1-z)}+\frac{z-1}{8}\artanh(z)+\frac{z}{8}
  \end{align*}
  where the final summand $z/8$ has to be added as we sum over $n\ge
  2$ only.
  Minor simplifications lead to the desired formula. The generating function for
  $\E{X_n^\searrow}$ follows by adding
  \begin{equation*}
    -\frac{z^2}{2(1-z)}+\frac{1+z}{2}\artanh(z)-\frac{z}{2},
  \end{equation*}
  corresponding to \eqref{eq:E_X_down_n_exact}.

  The asymptotic expressions follow from the explicit formul\ae{} via 
  Lemma~\ref{lem:harmonic-asy}.
\end{proof}
}


\part{Part III: Dual-Pivot Quicksort}
\label{sec:quicksort}

In this third and last part of this work, we finally analyze two
different partitioning strategies and the dual-pivot quicksort
algorithm itself.

As mentioned in the introduction, the number of comparisons of
dual-pivot quicksort depends on the concrete partitioning procedure.  
For example, if one wants to classify a
large element, i.\,e., an element larger than the larger pivot, 
comparing it with the larger pivot is unavoidable,
but it depends on the partitioning procedure whether
a comparison with the smaller pivot occurs, too. 

First, in Section~\ref{sec:solve-recurrence},
we make our set-up precise, fix notions, and start solving the
dual-pivot quicksort recurrence~\eqref{eq:recurrence}. This recurrence
relates the cost of the partitioning step to the total number of
comparisons of dual-pivot quicksort.

In Section~\ref{sec:part-costs} two partitioning strategies, called ``Clairvoyant'' and
``Count'', are introduced and their respective cost is analyzed. 
It will turn out that the results on lattice paths obtained in Parts~\ref{sec:lattice-paths} and
\ref{sec:more-lattice-paths} are central in determining the partitioning cost exactly.

Everything is put together in Section~\ref{sec:costs-main-asy}:
We obtain the exact comparison count for two versions of dual-pivot quicksort
(Theorems~\ref{thm:clairvoyant:cost} and~\ref{thm:count:cost}). 
The asymptotic behavior is extracted out of the exact results
(Corollaries \ref{cor:clairvoyant:cost:asy} and~\ref{cor:count:cost:asy}).


\section{Solving the Dual-Pivot Quicksort Recurrence}
\label{sec:solve-recurrence}

We consider versions of dual-pivot quicksort that act as follows
on an input sequence $(a_1,\ldots,a_n)$ consisting of distinct numbers:
If $n\le1$, do nothing, otherwise choose $a_1$ and $a_n$ as pivots,
and by one comparison determine $p=\min(a_1,a_n)$ and $q=\max(a_1,a_n)$.
Use a partitioning procedure to partition the remaining
$n-2$ elements into the three classes \emph{small}, \emph{medium}, and
\emph{large}. Then call dual-pivot quicksort recursively on each of these three classes
to finish the sorting, using the same partitioning procedure in all recursive calls.  

Let $P_n$, a random 
variable, denote the \emph{partitioning cost}. This is 
defined as the number of comparisons made by the partitioning procedure
if the input $(a_1,\ldots,a_n)$ is assumed to be in random order.
Further, let $C_n$ be the random variable
that denotes the number of comparisons carried out when sorting $n$ elements
with dual-pivot quicksort.
The reader should be aware that both $P_n$ and $C_n$ are determined by the partitioning
procedure used.  

Since the input $(a_1,\ldots,a_n)$ is in random order
and the partitioning procedure does nothing but compare elements with the 
two pivots, the inputs for the recursive calls are in random order as well,
which implies that the distributions of $P_n$ and $C_n$
only depend on $n$. In particular we may assume that when the sorting algorithm
is called on $n$ elements during recursion, 
the input is a permutation of $\{1,\ldots,n\}$.

The recurrence
\begin{equation}
  \E{C_n} =
     \E{P_n}  + \frac{3}{\binom{n}{2}}
      \sum_{k = 1}^{n - 2} (n - 1 - k) \E{C_{k}}
\label{eq:recurrence}
\end{equation}
for $n\ge0$
describes the connection between
the expected sorting cost $\E{C_n}$ and the expected partitioning cost $\E{P_n}$.
It will be central for our analysis. 
Note that it is irrelevant for \eqref{eq:recurrence} how the partitioning cost $\E{P_n}$ is determined;
it need not even be referring to comparisons. 
The recurrence is simple and well-known; a version of it occurs already in 
Sedgewick's thesis~\cite{Sedgewick:1975:thesis}.
For the convenience of the reader we give a brief justification in 
Appendix~\ref*{sec:appendix:quicksort:recurrence}. 
\def\appendixqsrecurrence{
We justify \eqref{eq:recurrence}. Assume the input is a random permutation of $\{1,\ldots,n\}$. 
The expected cost $\E{C_n}$ is the sum of the expected partitioning cost $\E{P_n}$
and the sum of the costs of the recursive calls for the small, the medium, and the large elements.
A recursive call for a group of size $k$ has expected cost $\E{C_k}$, for $0\le k \le n-2$.
The probability that the number of small elements is exactly $k$ is $(n-1-k)/\binom{n}{2}$,
since there are $\binom{n}{2}$ possible pivot pairs $\{p,q\}$ in total and $n-1-k$ pairs $\{p,q\}$ with $p=k+1$.
So the expected contribution to $\E{C_n}$ from the set of small elements 
is $\sum_{1\le k \le n-2}((n-1-k)/\binom{n}{2})\E{C_k}$.
Since there are $n-1-k$ pivot pairs with $k=q-p-1$ and 
$n-1-k$ pivot pairs with $q=n-k$, the contributions from the
recursive calls for the set of medium and the set of large elements are the same.
Adding the three contributions we obtain \eqref{eq:recurrence}.
}
In Hennequin~\cite{Hennequin:1991:analy} recurrence \eqref{eq:recurrence}
was solved exactly for $\E{P_n} =  a n + b$, where $a$ and $b$ are constants.
For $\E{P_n} =  a n + O(n^{1-\varepsilon})$
the solution is $\E{C_n} = \frac65 an \log n + O(n)$, see \cite[Theorem 1]{AumullerD15}.

\def\appendixrecint{
We recall how to solve recurrence~\eqref{eq:recurrence} using
generating functions. We follow 
Wild~\cite[\S~4.2.2]{Wild2013} who in turn follows
Hennequin~\cite{Hennequin:1991:analy}. The following lemma is contained in
slightly different notation in \cite{Wild2013}; nevertheless, we include the
proof here for the sake of self-containedness. This also allows us to make the
integration bounds explicit.

\begin{lemma}\label{le:integration}With $C_n$ and $P_n$ as above,
  $C(z)=\sum_{n\ge 0}\E{C_n}z^n$ and $P(z)=\sum_{n\ge 0}\E{P_n}z^n$, we have
  \begin{equation*}
    C(z)=(1-z)^3\int_{0}^z (1-t)^{-6}\int_{0}^t (1-s)^3P''(s)\,ds\,dt.
  \end{equation*}
\end{lemma}
\begin{proof}
Multiplying \eqref{eq:recurrence} by $n(n-1)z^{n-2}$ and summing over all $n\ge
2$ yields
\begin{multline*}
  \sum_{n\ge 2}n(n-1)\E{C_n}z^{n-2}\\=\sum_{n\ge 2}n(n-1)\E{P_n}z^{n-2} + 6
  \sum_{n\ge 1}\sum_{k=0}^{n-1}(n-1-k)z^{n-k-2} \E{C_k}z^k.
\end{multline*}
Note that the range of the summations has been extended without any
consequences because of $\E{C_0}=0$. We replace $n-1$ by $n$ in the double sum
and write it as a product of
two generating functions:
\begin{multline*}
  \sum_{n\ge 1}\sum_{k=0}^{n-1}(n-1-k)z^{n-k-2} \E{C_k}z^k = 
  \sum_{n\ge 0}\sum_{k=0}^{n}(n-k)z^{n-k-1} \E{C_k}z^k \\= \biggl(\sum_{n\ge
    0}nz^{n-1}\biggr)C(z)=\biggl(\sum_{n\ge 0}z^n\biggr)'
  C(z)=\Bigl(\frac1{1-z}\Bigr)' C(z)=\frac{C(z)}{(1-z)^2}.
\end{multline*}
Thus we obtained
\begin{equation*}
  C''(z)=P''(z)+\frac{6}{(1-z)^2} C(z)
\end{equation*}
or, equivalently,
\begin{equation*}
  (1-z)^2C''(z)- 6 C(z)=(1-z)^2P''(z).
\end{equation*}
Setting $(\theta f)(z)=(1-z)f'(z)$ for a function $f$, this can be rewritten as
\begin{equation*}
  ((\theta^2+\theta-6)C)(z) = (1-z)^2P''(z).
\end{equation*}
Factoring $\theta^2+\theta-6$ as $(\theta-2)(\theta+3)$ and setting
$D=(\theta+3)C$, we first have to solve
\begin{equation*}
  ((\theta-2)D)(z)=(1-z)^2P''(z),
\end{equation*}
i.\,e.,
\begin{equation*}
  (1-z)D'(z)-2D(z)=(1-z)^2 P''(z).
\end{equation*}
Multiplication by $(1-z)$ yields
\begin{equation*}
  \bigl((1-z)^2 D(z)\bigr)'=(1-z)^3 P''(z).
\end{equation*}
Integration and the fact that $D(0)=C'(0)+3C(0)=\E{C_1+3C_0}=0$ yields
\begin{equation*}
  D(z)=\frac1{(1-z)^2}\int_0^z (1-s)^3P''(s)\,ds.
\end{equation*}
We still have to solve
\begin{equation*}
  (1-z)C'(z)+3C(z)=D(z).
\end{equation*}
We multiply by $(1-z)^{-4}$ and obtain
\begin{equation*}
  \bigl((1-z)^{-3}C(z)\bigr)'=(1-z)^{-4}D(z).
\end{equation*}
As $C(0)=0$, we obtain
\begin{equation*}
  C(z)=(1-z)^3\int_{0}^z (1-t)^{-4} D(t)\, dt.
\end{equation*}
\end{proof}
}


\section{Partitioning Algorithms and Their Cost}\label{sec:part-costs}

In Section~\ref{sec:solve-recurrence} we saw that
in order to calculate the average number of
comparisons~$\E{C_n}$ of a dual-pivot quicksort algorithm
we need the expected partitioning cost~$\E{P_n}$
of the partitioning procedure used. 
The aim of this section is to determine $\E{P_n}$ 
for two such partitioning procedures, ``Clairvoyant'' and ``Count'', 
to be described below.

We use the set-up described at the beginning of
Section~\ref{sec:solve-recurrence}.  
For partitioning we use comparisons to \emph{classify}
the $n-2$ elements $a_2$, \dots, $a_{n-1}$ as \emph{small},
\emph{medium}, or \emph{large}. We will be using the term \emph{classification} 
for this central aspect of partitioning. Details of a partitioning 
procedure that concern how the classes are
represented or elements are moved around may and will be ignored. (Nonetheless,
in Appendix~\ref*{sec:pseudocode}
we provide pseudocode for the considered classification strategies turned into 
dual-pivot quicksort algorithms.)
The cost $P_n$ depends on the concrete classification strategy used, 
the only relevant difference between
classification strategies being whether 
the next element to be classified is compared with
the smaller pivot~$p$ or the
larger pivot~$q$ first. This decision may
depend on the whole history of outcomes of previous comparisons.  (The
resulting abstract classification strategies may conveniently be
described as classification trees, see~\cite{AumullerD15}, but we do
not need this model here.)

Two comparisons are necessary for each medium element.
Furthermore, one comparison with $p$ is necessary for small and one
comparison with $q$ is necessary for large elements. As the input
consists of the elements $1$, \dots, $n$, there are $p - 1$ small,
$q - p - 1$ medium, and $n - q$ large elements.  Averaging over all
$\binom{n}{2}$ positions of the pivots, we see that on average
\begin{equation}\label{eq:part-necessary}
  \frac43(n-2) + 1
\end{equation}
\emph{necessary comparisons} are required no matter how the classification procedure
works, see \cite[(5)]{AumullerD15}; the summand $+1$ corresponds to
the comparison of $a_1$ and $a_n$ when choosing the two pivots.

We call other comparisons occurring during classification \emph{additional comparisons}.
That means, an additional comparison arises when a small element is compared with $q$ first or a large element
is compared with $p$ first.
In order to obtain $\E{P_n}$ for some classification strategy, we
have to calculate the expected number of additional comparisons.

Next we describe two (closely related) classification strategies
from~\cite{AumullerD15}.
Let $s_i$ and $\ell_i$ denote the number of elements that have been
classified as small and large, respectively, in the first $i$ classification rounds. Set $s_0 = \ell_0 = 0$.

\begin{strategy}[``Clairvoyant'']
  Assume the input contains $s=p-1$
  small and $\ell=n-q$ large elements.  When classifying the $i$th
  element, for $1\le i \le n-2$, proceed as follows: If $s - s_{i - 1}
  \geq \ell - \ell_{i - 1}$, compare with $p$ first, otherwise
  compare with $q$ first.
\end{strategy}
The number of additional comparisons of Clairvoyant is
denoted by $\Acv_n$, its partitioning cost $\Pcv_n$. 

Note that the strategy ``Clairvoyant'' cannot be implemented
algorithmically, since $s$ and $\ell$ are not known until the
classification is completed.

As shown in \cite[Section 6]{AumullerD15}, this strategy offers the
smallest expected classification cost among all strategies that have
oracle access to $s$ and $\ell$ at the outset of a classification
round. As such, its expected cost is a lower bound for the cost of all
algorithmic classification procedures; hence we call it an \emph{optimal strategy}.

The non-algorithmic strategy ``Clairvoyant'' can be turned into an
algorithmic classification strategy, which is described next. It will turn
out that its
cost is only marginally larger than 
that of strategy ``Clairvoyant''.

\begin{strategy}[``Count'']
  When classifying the $i$th element, for $1\le i \le n-2$, proceed as
  follows: If $s_{i - 1} \geq \ell_{i - 1}$, compare with $p$
  first, otherwise compare with $q$ first.
\end{strategy}
The number of additional comparisons of this strategy is
called $\Act_n$, its cost $\Pct_n$.

No algorithmic solution for the classification 
problem can have cost smaller than ``Clairvoyant''.
Strategy ``Count'' is algorithmic.
Thus any cost-minimal algorithmic classification procedure has cost between
$\E{\Pcv_n}$ and $\E{\Pct_n}$, and a precise analysis of both
will lead to good lower and upper bounds for the cost of such a procedure.
It was shown in \cite{AumullerD15} that $\E{\Pct_n} -
\E{\Pcv_n} = O(\log n)$ and that, as a consequence, both
strategies lead to dual-pivot quicksort algorithms that use $\frac95n
\log n + O(n)$ comparisons on average. 
In the following, we carry out a precise analysis of $\E{\Pcv_n}$ and
$\E{\Pct_n}$, which will make it possible to determine the
expected comparison count of an optimal dual-pivot quicksort algorithm
up to $0.28n$.
	
\begin{lemma}\label{lem:additional:comparisons:clairvoyant:count}
	\begin{enumerate}[(a)]
  \item The expected number of additional comparisons of strategy
    ``Clairvoyant'' is
    \begin{equation*}
      \E{\Acv_n} = \frac{n}{6} - \frac{7}{12} + \frac{1}{4(n-\iverson*{$n$ even})}
      - \E{X^\searrow_n}.
    \end{equation*}
    
  \item The expected number of additional comparisons of strategy
    ``Count'' is
    \begin{equation*}
      \E{\Act_n} = \frac{n}{6} - \frac{7}{12} + \frac{1}{4(n-[n\text{ even}])}
      + \E{X^\nearrow_{n}}.
    \end{equation*}
  \end{enumerate}
\end{lemma}

\begin{proof}[ideas]
(The full proof can be found in Appendix~\ref*{sec:appendix:part-costs}. 
A different proof of a related statement was given in~\cite{AumullerD15}.)

(a) Noticing that medium elements can be ignored,
we consider a reduced input of size $n' = s + \ell$,
consisting only of the $s$ small and the $\ell$ large elements in the input.
For $0 \leq i \leq n'$ let $s_i' = s - s_i$ and $\ell_i' = \ell - \ell_i$ 
denote the number of small respectively large elements
left unclassified after step $i$. 
Then $\{(i,s_i'-\ell_i')\mid 0 \leq i \leq n'\}$ is a lattice
path with distribution (including the distribution of $n'$)
exactly as in Section~\ref{sec:lattice-paths-N},
so that the results on the expected number of zeros on such paths given there may be applied.
We also note that the sign of $s_{i-1}'-\ell_{i-1}'$ decides whether
the $i$th element to be classified is compared with $p$ first or with $q$ first, 
and that additional comparisons correspond to steps on the path that 
lead away from the horizontal axis, excepting down-from-zero steps
(due to the asymmetry in treating the situation $s-s_i = \ell - \ell_i$
in strategy ``Clairvoyant''). For the number of steps away from the
horizontal axis one easily finds the expression $\min(s,\ell)$.
Averaging over all choices for $n'$ and the two pivots leads to the formula claimed in (a).

(b) Now assume strategy ``Count'' is applied to 
$n' = s + \ell$ elements. 
The set $\{(i,s_i-\ell_i) \mid 0 \leq i \leq n'\}$
forms a lattice path that starts at $(0,0)$ and ends at $(n',s-\ell)$.
It can be shown that reflection with respect to the vertical line at $n'/2$ 
maps these paths in a probability-preserving way to the paths from from (a) (and thus from our model), 
and it turns out that additional comparisons in this strategy correspond to steps away from
the horizontal axis and up-to-zero steps.
As in (a), averaging leads to the formula claimed in (b).
\end{proof}

\def\appendixpartcost{
\begin{proof}
We start with Part (a). A different proof of a related statement was given in~\cite{AumullerD15}.
Since the pivots are $p$ and $q$ with $p<q$,
there are $s=p-1$ small, $\ell=n-q$ large, and $p-q-1$ medium elements.
Omit all medium elements (which require two comparisons) to obtain a reduced sequence
$(b_1,\dots,b_{n'})$ of elements to be classified, with $n'=s+\ell$.
Note that the distribution of $n'$ is exactly that of the random variable $N'$
in Section~\ref{sec:lattice-paths-N}.
Further, if $n'$ is given, $s$ is uniform in $\{0,\ldots,n'\}$
and $s-\ell$ is uniform in $\{x\in\Z\mid |x|\le n', x\equiv n'\pmod 2\}$.
By the assumption that the input is in random order
it is irrelevant in which order the elements are classified; 
so we may assume the order is $b_1,\dots,b_{n'}$.
For $0 \leq i \leq n'$ let $s_i'$ and $\ell_i'$ 
denote the number of small respectively large elements
in $(b_{i + 1}, \ldots, b_{n'})$,
and let $d_i=s_i' - \ell_i'$.
Then the sequence 
\begin{equation*}
  ((0, d_0), (1, d_1), \ldots, (n', 0))
\end{equation*}
is a lattice path of length $n'$ as defined in Section~\ref{sec:description}.
(Figure~\ref{fig:qs:lattice:walk:ex} depicts such a path.)
When $s$ and $\ell$ are given, the path is determined by the 
random order of the small and large elements in $(b_1,\dots,b_{n'})$, 
and hence all these paths have the same probability $1/\binom{s+\ell}{s}$.
We now note that the path contains the full information about the
additional comparisons carried out by strategy ``Clairvoyant'':
For $1\le i \le n'$, an edge from $(i-1,d_{i-1})$ to $(i,d_i)$ will correspond to an additional
comparison occurring when classifying $b_i$ if and only if either
\begin{itemize}
	\item $d_{i-1} = s_{i-1}' - \ell_{i-1}' \ge 0$ (so that $b_i$ is compared with $p$ first)
and $d_i > d_{i-1}$ (which means that $b_i$ is a large element), or
\item $d_{i-1} = s_{i-1}' - \ell_{i-1}' < 0$ (so that $b_i$ is compared with $q$ first) 
and $d_i < d_{i-1}$ (which means that $b_i$ is a small element).
\end{itemize}
Thus, if we let
\begin{align*}
\#\mathrm{up}_{\ge0} &= \#(\text{steps going up, starting on or above the horizontal axis}),\\
\#\mathrm{down}_{<0} &= \#(\text{steps going down, starting strictly below the horizontal axis}),
\end{align*}
then the additional cost equals $\#\mathrm{up}_{\ge0} + \#\mathrm{down}_{<0}$.

Now assume $s\ge \ell$ (see Figure~\ref{fig:qs:lattice:walk:ex}), and let
\begin{align*}
\#\mathrm{down}_{>0} &= \#(\text{steps going down, starting strictly above the horizontal axis}).
\end{align*}
Since the path ends at $(n',0)$ and up and down steps below the horizontal axis cancel, 
we have $s - \ell = d_ 0 =  \#\mathrm{down}_{>0}  - \#\mathrm{up}_{\ge0}$, hence 
$\#\mathrm{up}_{\ge0} =   \#\mathrm{down}_{>0} - (s - \ell)$.
Thus, the additional cost is 
\begin{equation*}
  \#\mathrm{down}_{>0} + \#\mathrm{down}_{<0} -(s-\ell).
\end{equation*}
Let $z_{n'}^\searrow$ be the number of steps that go down-from-zero 
(cf.\@ Appendix~\ref*{sec:appendix:more-zeros}).
The total number of steps that go down is 
$s = \#\mathrm{down}_{>0} + \#\mathrm{down}_{<0} + z_{n'}^\searrow$,
so that the additional cost turns out to be
\begin{equation*}
  s - z_{n'}^\searrow - (s-\ell) = \ell - z_{n'}^\searrow.
\end{equation*}
In a similar way one sees that the assumption $\ell > s$ leads to  $s - z_{n'}^\searrow$ additional comparisons.
Combining both cases, the number of additional comparisons is $\min(s,\ell) - z_{n'}^\searrow$.
Recalling the definition of $X_{n}^\searrow$ from Section~\ref{sec:lattice-paths-N}, 
averaging over all possible pivots $p<q$, and using $s=p-1$ and $\ell=n-q$, 
we obtain
\begin{equation*}
      \E{\Acv_n} = \frac{1}{\binom{n}{2}}
      \biggl(\sum_{1 \leq p < q \leq n} \min(p-1, n - q)\biggr)
      - \E{X^\searrow_n}.
    \end{equation*}
The first summand evaluates to 
$
\frac{n}{6} - \frac{7}{12} + \frac{1}{4(n-[n\text{ even}])}
$. 


\begin{figure}
\scalebox{0.85}{
\begin{tikzpicture}[yscale=1,xscale=0.8]
    \draw[->] (0, 0) to (0, 6.25);
    \draw[->] (0, 3) to (16.5, 3);
    \node at (0.5, 2.75) {$1$};
    \node at (15.5, 2.75) {$n'$};
    \node (l2) at (16.75, 3) {$i$};
    \node (l1) at (-0, 6.5) {$s'_i - \ell'_i$};

    \foreach \x in {1, 2, 3, 4, 5, 6, 7, 8, 9, 10, 11, 12, 13, 14, 15, 16}
        \draw (\x, 2.85) to (\x, 3.15);
    
        \foreach \x/\label in {0/-3, 1/-2, 2/-1, 3/\phantom{-}0, 4/\phantom{-}1, 5/\phantom{-}2, 6/\phantom{-}3}
    \draw (-0.1, \x) to node[xshift=-0.5cm] {$\label$} (0.1, \x);

    \foreach [count = \cnt] \p in {(1,3),(2,2), (3,1), (4,0), (5, -1), (6, 0), (7, - 1), (8, -2), (9, -1), 
                (10, 0), (11, 1), (12, 0), (13, 1), (14, 2), (15, 1), (16, 0)}
                \node[yshift=3cm, circle, fill, inner sep = 0, minimum size = 0.15cm] (p\cnt) at \p {};

                \draw (0, 5) to node[pos=0.5, diamond, draw, dotted, fill=white, inner sep = 0.05cm] {$\underline{\lambda}$} (p1);

                \foreach \lastpt/\pt/\label/\shape in {
                  1/2/$\sigma$/, 2/3/$\sigma$/, 3/4/$\sigma$/, 4/5/$\sigma$/draw, 5/6/$\lambda$/draw,
                  6/7/$\sigma$/draw, 7/8/$\underline{\sigma}$/draw, 8/9/$\lambda$/, 9/10/$\lambda$/draw,
                  10/11/$\underline{\lambda}$/draw, 11/12/$\sigma$/,
                  12/13/$\underline{\lambda}$/draw, 13/14/$\underline{\lambda}$/draw,
                  14/15/$\sigma$/, 15/16/$\sigma$/}
    \draw (p\lastpt) to node[pos=0.5, fill=white, diamond, \shape,
                             dotted, inner sep=0.05cm]
                  {\label}  (p\pt) ;
\end{tikzpicture}}
    \caption{Example run of strategy ``Clairvoyant'' and ``Count'' on input ($\lambda$, $\sigma$, $\sigma$, $\sigma$, $\sigma$, $\lambda$, $\sigma$, $\sigma$, $\lambda$, $\lambda$, $\lambda$, $\sigma$, $\lambda$, $\lambda$, $\sigma$, $\sigma$) and the reversed input, respectively. The $s=9$ small elements 
		in $(b_1,\ldots,b_{16})$ are represented 
		as ``$\sigma$'', the $\ell=7$ large ones as ``$\lambda$''. 
		The horizontal axis shows the input position $i$, the vertical axis shows the difference $d_i=s_i' - \ell_i'$ of small and large elements in the part $(b_{i+1},\ldots,b_{16})$ of the input. Underlined entries mark elements where an additional comparison occurs when strategy ``Clairvoyant'' is used from left to right. 
This strategy makes five additional comparisons on this input. 
  A diamond marks elements where an additional comparison occurs when strategy ``Count'' is used when treating the input from right to left. This strategy makes nine additional comparisons on this reversed input.}
\label{fig:qs:lattice:walk:ex}
\end{figure}




We continue with Part (b).  
Assume that pivots $p$ and $q$ and a reduced input $(b_1, \dots,b_{n'})$ 
is produced as in Part (a). 
For analyzing the classification strategy ``Count'', we wish to 
utilize our knowledge about ``Clairvoyant''. To this end, we use a reflection
trick and assume the input is treated in the reverse order $b_{n'}, \dots,b_1$.
(This is harmless since the order in which the elements are treated is irrelevant anyway.)
As in the specification of ``Count'', 
let $s_i$ and $\ell_i$ denote the number of small and large elements that have been seen
after the $i$th classification step.
Here this is the number of small and large elements in $\{b_{n'-i+1},\dots,b_{n'}\}$;
hence we have $s_i=s_{n'-i}'$ and $\ell_i=\ell_{n'-i}'$, and $s_i -\ell_i = d_{n'-i}$. 
Keeping track of $s_i-\ell_i$ for $i=0,\dots,n'$ gives rise to a lattice path,
which we imagine as running in the opposite direction from our standard paths:
It starts at $(n',0)$ and ends at $(0,s-\ell)$. The point reached after step $i$ 
is $(n' - i, s_i - \ell_i)=(n' - i,d_{n'-i})$. 
(An illustration is obtained by traversing the path 
in Figure~\ref{fig:qs:lattice:walk:ex} from right to left.)
Thus, the path induced by applying ``Clairvoyant'' to $(b_1, \dots,b_{n'})$ 
and the path induced by applying ``Count'' to $(b_{n'}, \dots,b_1)$ are the same,
they are just traversed in opposite directions.
The probability with which a path appears is the same in both situations.

Now we wish to determine the additional cost of the run of ``Count'' on
$(b_{n'}, \dots,b_1)$. 
For $1\le i \le n'$, reading from right to left,
an edge from $(n'-i+1,s_{n'-i+1}-\ell_{n'-i+1})$ to $(n'-i,s_{n'-i}-\ell_{n'-i})$ 
will correspond to an additional comparison needed when classifying $b_{n'-i+1}$ if and only if either
\begin{itemize}
	\item $ 0 \le s_{i-1} - \ell_{i-1} = d_{n'-i+1}$ (so that $b_{n' - i + 1}$ is compared with $p$ first)
and $d_{n'-i} < d_{n'-i+1}$ (which means that $b_{n'-i + 1}$ is a large element), or
\item $ 0 > s_{i-1} - \ell_{i-1}  = d_{n'-i+1}$ (so that $b_{n' - i + 1}$ is compared with $q$ first)
and $d_{n'-i} > d_{n'-i+1}$ (which means that $b_{n'-i + 1}$ is a small element).
\end{itemize}
We now interpret these relations in the standard direction
from left to right. Additional comparisons correspond 
to steps in the path 
\begin{itemize}
\item from $(n'-i,d_{n'-i})$ to $(n'-i+1,d_{n'-i+1})$ with $d_{n'-i+1} \ge 0$
and $d_{n'-i} < d_{n'-i+1}$ and 
\item from $(n'-i,d_{n'-i})$ to $(n'-i+1,d_{n'-i+1})$ with $d_{n'-i+1} < 0$
and $d_{n'-i} > d_{n'-i+1}$.
\end{itemize}
If we let
\begin{align*}
\#\mathrm{up}_{\ge-1} &= \#(\text{steps going up, ending on or above the horizontal axis}),\\
\#\mathrm{down}_{\le0} &= \#(\text{steps going down, starting on or below the horizontal axis}),
\end{align*}
then the additional cost equals $\#\mathrm{up}_{\ge-1} + \#\mathrm{down}_{<0}$, see
Figure~\ref{fig:qs:lattice:walk:ex} for an example.

Comparing with what we obtained in (a) for ``Clairvoyant'' we see that 
the additional cost of ``Count'' is the additional cost of ``Clairvoyant''
plus the number of steps from $(n'-i,-1)$ to $(n'-i+1,0)$ (in the notation of 
Sections \ref{sec:more-zeros} and~\ref{sec:lattice-paths-N} 
these are ``up-to-zero steps'', their number being denoted by $z_{n'}^\nearrow$)
and the number of steps from $(n'-i,0)$ to $(n'-i+1,-1)$
(these are ``down-from-zero steps'', their number being denoted by $z_{n'}^\searrow$).
Since the additional cost for ``Clairvoyant'' was
$ \min(p-1,n-q) - z_{n'}^\searrow$, for ``Count'' we get a cost of 
\begin{equation*}
  \min(p-1,n-q) + z_{n'}^\nearrow.
\end{equation*}
Averaging over all $p<q$ as in (a) finally yields
\begin{equation*}
      \E{\Act_n} = \frac{1}{\binom{n}{2}}
      \biggl(\sum_{1 \leq p < q \leq n} \min(p-1, n - q)\biggr)
      + \E{X^\nearrow_n},
    \end{equation*}
and evaluating the sum as in Part (a) finishes the proof.
\end{proof}

}

Lemma~\ref{lem:additional:comparisons:clairvoyant:count} allows  us to 
give an exact expression for the average number of comparisons of ``Clairvoyant'' 
and ``Count'' in a single partitioning step.
The expressions for $\E{\Pcv_n}$ and $\E{\Pcv_n}$ are obtained by adding the expected number of necessary comparisons $\frac43 (n - 2) + 1$ to the cost terms in Lemma~\ref{lem:additional:comparisons:clairvoyant:count} (see Appendix~\ref*{sec:appendix:part-costs}).

\def\appendixpcost{
\begin{lemma}\label{lem:expected:comparisons:clairvoyant:count}
Let $n\ge 2$.
\begin{enumerate}[(a)]
  \item The expected number of comparisons of strategy ``Clairvoyant''
  is
  \begin{equation*}
    \E{\Pcv_n} =
    \frac32 n
    - \frac{9}{4}
    + \frac{1}{4(n - \iverson*{$n$ even})}
    - \E{X^\searrow_n}.
  \end{equation*}

  \item The expected number of comparisons of strategy ``Count'' is
    \begin{align}
      \E{\Pct_n} = \frac32 n
      - \frac{9}{4}
      + \frac{1}{4(n - \iverson*{$n$ even})}
      + \E{X^\nearrow_{n}}.
      \label{eq:comparisons:count:down:from:zero}
    \end{align}
\end{enumerate}
\end{lemma}

\begin{proof}
The expressions for $\E{\Pcv_n}$ and $\E{\Pcv_n}$ are obtained by adding the expected number of necessary comparisons $\frac43 (n - 2) + 1$ to the cost terms in Lemma~\ref{lem:additional:comparisons:clairvoyant:count}.
\end{proof}
}


\section{Main Results and their Asymptotic Aspects}
\label{sec:costs-main-asy}

In this section we give precise formulations of our main results. 
We use the partitioning cost from the previous section to
calculate the expected number of comparisons of the two dual-pivot quicksort variants
obtained by using classification strategies ``Clairvoyant'' and ``Count'', respectively.
We call these sorting algorithms ``Clairvoyant'' and ``Count'' again. Recall
that ``Clairvoyant'' uses an oracle and is comparison-optimal,
and that ``Count'' is its algorithmic version. We validated our main results
in experiments which can be found in Appendix~\ref*{sec:empirical:validation}.
They show that the error term $\Oh[empty]{n^{-4}}$ is small
already for real-life input sizes $n$, and that the linear term
has a big influence even for larger $n$.

\begin{theorem}\label{thm:clairvoyant:cost}
  For $n\ge 4$, the average number of comparisons in the comparison-optimal dual-pivot
  quicksort algorithm ``Clairvoyant'' (with oracle) is
  \begin{equation*}
    \E{\Ccv_n} = 
    \frac{9}{5}nH_n + \frac{1}{5}n\Halt_n -\frac{89}{25}n + \frac{77}{40}H_n+\frac{3}{40}\Halt_n+\frac{67}{800}-\frac{(-1)^n}{10}+ r_n
    \end{equation*}
   where
   \begin{equation*}
     r_n = \frac{\iverson*{$n$ even}}{320}\Bigl(\frac{1}{n-3}+\frac{3}{n-1}\Bigr)-\frac{\iverson*{$n$ odd}}{320}\Bigl(\frac{3}{n-2}+\frac{1}{n}\Bigr).
   \end{equation*}
\end{theorem}


\begin{corollary}\label{cor:clairvoyant:cost:asy}
  The average number of comparisons in the
  algorithm ``Clairvoyant'' is
  \begin{equation*}
    \E{\Ccv_n} = 
    \frac95n \log n + A n + B \log n + C
    + \frac{D}{n} + \frac{E}{n^2} + \frac{F\iverson*{$n$ even} + G}{n^3}
    + \Oh[Big]{\frac{1}{n^4}}
  \end{equation*}
  with
  \begin{align*}
    A &= \frac95\gamma
    - \frac{1}{5} \log 2
    - \frac{89}{25}
    = -2.6596412392892\dots, &
    B &= \frac{77}{40}
    = 1.925, \\
    C &= \frac{77}{40}\gamma
    - \frac{3}{40}\log 2
    + \frac{787}{800}
    = 2.042904116393455\dots, &
    D &= \frac{13}{16} = 0.8125, \\
    E &= - \frac{77}{480} = -0.1604166\dots,\qquad
    F = \frac{1}{8} = 0.125, &
    G &= - \frac{19}{400} = -0.0475,
  \end{align*}
  asymptotically as $n$ tends to infinity.
\end{corollary}

Before continuing with the second partitioning strategy, let us
make a remark on the (non-)influence of the parity of~$n$. It is noteworthy
that in Corollary~\ref{cor:clairvoyant:cost:asy} no such
influence is visible in the first six terms (down to $1/n^2$); only
from $1/n^3$ on the parity of $n$ appears. This is somewhat unexpected, since a term $(-1)^n$ appears 
in Theorem~\ref{thm:clairvoyant:cost}.


\begin{theorem}\label{thm:count:cost}
  The average number of comparisons in the dual-pivot quicksort
  algorithm ``Count'' is
  \begin{equation*}
    \E{\Cct_n} = 
    \frac{9}{5}nH_n - \frac{1}{5}n\Halt_n -\frac{89}{25}n + \frac{67}{40}H_n-\frac{3}{40}\Halt_n-\frac{83}{800}+\frac{(-1)^n}{10}
    - r_n 
  \end{equation*}
  where $r_n$ is defined in Theorem~\ref{thm:clairvoyant:cost}.
\end{theorem}

Again, the asymptotic behavior follows from the exact result.

\begin{corollary}\label{cor:count:cost:asy}
  The average number of comparisons in the algorithm ``Count'' is
  \begin{equation*}
    \E{\Cct_n} = 
    \frac95n \log n + A n + B \log n + C
    + \frac{D}{n} + \frac{E}{n^2} + \frac{F\iverson*{$n$ even} + G}{n^3}
    + \Oh[Big]{\frac{1}{n^4}}
  \end{equation*}
  with
  \begin{align*}
    A &= \frac95\gamma
    + \frac{1}{5} \log 2
    - \frac{89}{25}
    = -2.3823823670652\dots, &
    B &= \frac{67}{40}
    = 1.675, \\
    C &= \frac{67}{40}\gamma
    + \frac{3}{40}\log 2
    + \frac{637}{800}
    = 1.81507227725206\dots, &
    D &= \frac{11}{16} = 0.6875, \\
    E &= - \frac{67}{480} = -0.1395833\dots,\qquad
    F = - \frac{1}{8} = -0.125, &
    G &= \frac{31}{400} = 0.0775,
  \end{align*}
  asymptotically as $n$ tends to infinity.
\end{corollary}
The idea of the proofs of Theorems~\ref{thm:clairvoyant:cost} and
\ref{thm:count:cost} is to translate the recurrence~\eqref{eq:recurrence} into
a second order differential equation for the generating function $C(z)$ of $\E{C_n}$
in terms of the generating function $P(z)$ of $\E{P_n}$. Integrating twice
yields $C(z)$. This generating function then allows extraction of the exact
expressions for $\E{C_n}$. The asymptotic results follow. See
Appendix~\ref*{sec:appendix:costs-main-asy} for details.

\def\appendixproofmain{
\begin{proof}[of Theorems~\ref{thm:clairvoyant:cost}
  and~\ref{thm:count:cost}]
  The partitioning cost of strategy
  ``Clairvoyant'' is stated in
  Lemma~\ref{lem:expected:comparisons:clairvoyant:count}.
  The corresponding generating functions can be obtained by using 
  Proposition~\ref{pro:dual-pivot-expect-exact} and \eqref{eq:artanh}:
  \begin{multline*}
    \Pcv(z)=\sum_{n\ge 2} \E{\Pcv_n} z^n 
    = \frac{3}{2(1-z)^2} -\frac{\artanh(z)}{2(1-z)} \\
    -\frac{25z^2}{8(1-z)} +\frac{3+z}{8}\artanh(z) -\frac32 -\frac{23z}{8}.
  \end{multline*}

  We calculate the comparison cost from the partitioning cost by means
  of Lemma~\ref{le:integration} and obtain
  \begin{multline*}
    \Ccv(z) = 
    -2\frac{\log(1-z)}{(1-z)^2}-\frac{2\artanh(z)}{5  (1-z)^2}
    - \frac{44}{25  (1-z)^2} 
    + \frac{\artanh(z)}{4 (1-z)}
    + \frac{279}{160 (1-z)}\\
    -\frac{(1-z)^3}{320}\artanh(z) -\frac{2}{75} z^{3}
    + \frac{123}{1600}  z^{2} - \frac{113}{1600} z + \frac{13}{800}.
  \end{multline*}
  Taking into account that $\artanh(z)=(\log(1+z)-\log(1-z))/2$,
  \begin{align*}
    \sum_{m\ge 1}\Halt_m z^m&=-\frac{\log(1+z)}{(1-z)},\\
    \sum_{m\ge 1} H_m z^m&=-\frac{\log(1-z)}{(1-z)},\\
    \sum_{m\ge 1}m\Halt_m z^{m}&=z \Bigl(-\frac{\log(1+z)}{(1-z)}\Bigr)'
    =- \frac{\log(1 + z)}{{(1-z)}^{2}}+\frac{\log(1 + z)}{1-z} + \frac{1}{2  {(1 + z)}} - \frac{1}{2  {(1-z)}}, \\
    \sum_{m\ge 1}m H_m z^{m}&=z \Bigl(-\frac{\log(1-z)}{(1-z)}\Bigr)'
    = - \frac{\log(1-z)}{{(1-z)}^{2}}+ \frac{1}{{(1-z)}^{2}}+\frac{\log(1-z)}{1-z} - \frac{1}{1-z} ,
  \end{align*}
  as well as \eqref{eq:artanh}, we obtain the result.

  The proof concerning strategy ``Count'' is analogous to the proof of
  Theorem~\ref{thm:clairvoyant:cost} above.
  The corresponding generating functions are
  \begin{multline*}
    \Pct(z)=\sum_{n\ge 2} \E{\Pct_n} z^n
    = \frac{3}{2(1-z)^2} + \frac{\artanh(z)}{2  {(1-z)}} \\
    -\frac{31z^2}{8(1-z)}-\frac{3+z}{8}\artanh(z) - \frac{3}{2}-\frac{25z}{8}
  \end{multline*}
  and
  \begin{multline*}
    \Cct(z) = 
    - \frac{8\log(1-z)}{5(1-z)^2}+\frac{2\artanh(z)}{5 (1-z)^2}
    - \frac{44}{25  (1-z)^2}
    - \frac{\artanh(z)}{4 (1-z)}
    + \frac{281}{160 (1-z)}\\
    + \frac{(1-z)^3}{320}\artanh(z) + \frac{1}{150} z^{3}
    - \frac{27}{1600} z^{2} + \frac{17}{1600} z + \frac{3}{800}.
  \end{multline*}
  in this case.
\end{proof}

\begin{proof}[of Corollaries~\ref{cor:clairvoyant:cost:asy}
  and~\ref{cor:count:cost:asy}]
  Insert the expansions of Lemma~\ref{lem:harmonic-asy} into the explicit
  representations of Theorems~\ref{thm:clairvoyant:cost} and \ref{thm:count:cost}.
\end{proof}
}

\def\appendixexperiments{
\section{Empirical Validation} 
\label{sec:empirical:validation}

We implemented strategies ``Clairvoyant'' and ``Count'' as dual-pivot quicksort algorithms
in a straight-forward way in C\texttt{++}. Pseudocode of the algorithms is presented in Appendix~\ref*{sec:pseudocode}. 

For small input sizes of length $n \in \{2, \ldots, 12\}$ we enumerated all permutations of $\{1, \ldots, n\}$ and verified that the average number of comparisons (computed over all permutations) obtained from the experimental measurements equals the results from Theorem~\ref{thm:clairvoyant:cost} and
Theorem~\ref{thm:count:cost}, respectively.

For larger inputs, we sorted random permutations of $\{1, \ldots, n\}$ for $n = 2^k, k \in \{11, \ldots, 28\},$ and counted the comparisons needed to 
sort the input. For each input size, we sorted $400$ different inputs. Figure~\ref{fig:experiments} shows the measurements we got. From these measurements we conclude that the average comparison counts for sorting over a small number of inputs match the exact average counts from Theorem~\ref{thm:clairvoyant:cost} and Theorem~\ref{thm:count:cost}, resp., very well.

\begin{figure}
    \centering
\begin{tikzpicture}
  \begin{axis}[%
    xlabel={Items (log-scale)},
    ylabel={\# Comparisons / $n \log n$},
    height=7.5cm,
    width=12cm,
    domain = 11:28,
    legend style = { at = {(0.6,0.25)}, anchor=west, draw=none},
    cycle list name = black white,
    legend columns = 1
    ]
    \addplot coordinates { (11.0,1.45124) (12.0,1.48548) (13.0,1.50672) (14.0,1.52512) (15.0,1.54) (16.0,1.55729) (17.0,1.57937) (18.0,1.58742) (19.0,1.59973) (20.0,1.6079) (21.0,1.61674) (22.0,1.62617) (23.0,1.63222) (24.0,1.64094) (25.0,1.64933) (26.0,1.65046) (27.0,1.65689) (28.0,1.66467) };
    \addlegendentry{Clairvoyant};
    \addplot coordinates { (11.0,1.4875) (12.0,1.5179) (13.0,1.5388) (14.0,1.55407) (15.0,1.5662) (16.0,1.58367) (17.0,1.60125) (18.0,1.60875) (19.0,1.61982) (20.0,1.62791) (21.0,1.63594) (22.0,1.64434) (23.0,1.64863) (24.0,1.65718) (25.0,1.66473) (26.0,1.66636) (27.0,1.67144) (28.0,1.67879) };
    \addlegendentry{Count};
    \addplot coordinates { (11.0,1.45225) (12.0,1.48078) (13.0,1.5051) (14.0,1.52606) (15.0,1.54426) (16.0,1.56022) (17.0,1.57431) (18.0,1.58684) (19.0,1.59805) (20.0,1.60815) (21.0,1.61728) (22.0,1.62559) (23.0,1.63317) (24.0,1.64012) (25.0,1.64652) (26.0,1.65242) (27.0,1.65789) (28.0, 1.66296) };
    \addlegendentry{Clairvoyant (exact)}
    \addplot coordinates { (11.0,1.48847) (12.0,1.51404) (13.0,1.53584) (14.0,1.55461) (15.0,1.57092) (16.0,1.58521) (17.0,1.59783) (18.0,1.60906) (19.0,1.61911) (20.0,1.62815) (21.0,1.63633) (22.0,1.64377) (23.0,1.65056) (24.0,1.65679) (25.0,1.66252) (26.0,1.66781) (27.0,1.6727) (28.0,1.67725) };
    \addlegendentry{Count (exact)}
    \end{axis}
\end{tikzpicture}
\caption{Average comparison count (scaled by $n \log n$) needed to sort a random
    input of up to $2^{28}$ integers. We show the measurements we got for
    the comparison-optimal strategy ``Clairvoyant'' and its algorithmic variant
    ``Count'' together with the exact average comparison counts from
    Theorem~\ref{thm:clairvoyant:cost} and Theorem~\ref{thm:count:cost}. All data points
    from experiments are the average over 400 trials.}
\label{fig:experiments}
\end{figure}
}

\appendix
\section*{Appendix}
The appendices can be found at \href{http://arxiv.org/abs/1602.04031v1}{arXiv:1602.04031v1}.

{\footnotesize
\bibliographystyle{alpha}
\bibliography{lit}
}


\end{document}

\clearpage
\appendix


\section{Using the Generating Function Machinery}
\label{sec:gf}


\begin{theorem}\label{thm:paths-gf-zeros}
  For a randomly (as described in Section~\ref{sec:description}) chosen path of
  length~$n$, the expected number of zeros is
  \begin{equation*}
    \E{Z_n} =
    \frac{4}{n+1}
    \sum_{0\leq k < \ell < \ceil{n/2}} \frac{\binom{n}{k}}{\binom{n}{\ell}}
    + \iverson*{$n$ even} \frac{1}{n+1}
    \left(\frac{2^n}{\binom{n}{n/2}} - 1\right) + 1.
  \end{equation*}
\end{theorem}

The remaining part of this section is devoted to the proof of this theorem.
The main technique is to model our lattice paths by means of combinatorial
classes and generating functions. For simplicity of notation,
we denote a combinatorial class by its corresponding generating function.

\begin{figure}
  \centering
    \begin{tikzpicture}[scale=0.25, latticepath/.style={very thick}]

    \draw (-22,-3) -- (-22,14);
    \draw (-24,0) -- (25,0);

    \newcommand{\Cpath}[3][]{
      \draw[latticepath, #1] ($(0,0) + #2$) -- ($(4,0) + #2$) --
      ($(2,2.82842712474619) + #2$) -- cycle;
      \node at ($(2,0) + #2$) [above] {#3};
    }
    \newcommand{\Cpathmirr}[3][]{
      \draw[latticepath, #1] ($(0,0) + #2$) -- ($(4,0) + #2$) --
      ($(2,-2.82842712474619) + #2$) -- cycle;
      \node at ($(2,0) + #2$) [below] {#3};      
    }

    \newcommand{\Cpathdown}[2]{
      \Cpath{#1}{#2}
      \draw[latticepath] ($(4,0) + #1$) -- ($(5,-1) + #1$);
    }

    \node at (-22,10) [left] {$s$};
    \Cpathdown{(-22,10)}{$C$}
    \Cpathdown{(-17,9)}{$C$}

    \draw[dotted] (-11,7) -- (-7,3);
    \draw[latticepath] (-6,2) -- (-5,1);
    \Cpathdown{(-5,1)}{$C$}

    \node[circle,inner sep=1.5pt,fill] at (0,0) {};
    \node[] at (0,0) [below] {$u$};
    \draw[latticepath] (0,0) -- (1,1);
    \Cpath{(1,1)}{$C$}
    \draw[latticepath] (5,1) -- (6,0);
    \draw[latticepath] (0,0) -- (1,-1);    
    \Cpathmirr{(1,-1)}{$C$}
    \draw[latticepath] (5,-1) -- (6,0);

    \node[circle,inner sep=1.5pt,fill] at (6,0) {};
    \node[] at (6,0) [below] {$u$};
    \draw[latticepath] (6,0) -- (7,1);
    \Cpath{(7,1)}{$C$}
    \draw[latticepath] (11,1) -- (12,0);
    \draw[latticepath] (6,0) -- (7,-1);    
    \Cpathmirr{(7,-1)}{$C$}
    \draw[latticepath] (11,-1) -- (12,0);

    \node[circle,inner sep=1.5pt,fill] at (12,0) {};
    \node[] at (12,0) [below] {$u$};
    \draw[dotted] (12,2.414) -- (18,2.414);
    \draw[dotted] (12,-2.414) -- (18,-2.414);

    \node[circle,inner sep=1.5pt,fill] at (18,0) {};
    \node[] at (18,0) [below] {$u$};
    \draw[latticepath] (18,0) -- (19,1);
    \Cpath{(19,1)}{$C$}
    \draw[latticepath] (23,1) -- (24,0);
    \draw[latticepath] (18,0) -- (19,-1);    
    \Cpathmirr{(19,-1)}{$C$}
    \draw[latticepath] (23,-1) -- (24,0);

    \node at (24,0) [below right] {$n$};

  \end{tikzpicture}

  \caption{Decomposition of a lattice path for $s\geq0$ marking zeros.}
  \label{fig:decomp-path-zeros}
\end{figure}

Concerning the generating functions, we mark a step to the right by
the variable~$z$ and a zero (except the last) by~$u$. Note that we do
not mark the zero at $(n,0)$ for technical reasons; we'll take this
into account at the end by adding a $1$ to the final result.
Thus, the coefficient of $z^nu^{r-1}$ of
the function $Q_s(z,u)$ (the generating function of all paths starting in
$(0,s)$ and ending in some $(n,0)$) equals the number of paths of length~$n$
and exactly $r$ zeros.

We also need the following auxiliary function. The generating
function~$\f{C}{z}$ counts all Catalan paths, i.\,e., paths starting and ending
at the same height, but not going below it. This equals
\begin{equation*}
  \f{C}{z} = \frac{1-\sqrt{1-4z^2}}{2z^2}.
\end{equation*}

In Figure~\ref{fig:decomp-path-zeros}, we give a schematic decomposition of a
path from $(0, s)$ to $(n, 0)$ for non-negative $s$. This decomposition
translates to the following parts of the generating function (the path is read
from the left to the right).
\begin{itemize}
\item We start by $s$ consecutive blocks of $\f{C}{z}$, each followed by a single
  descent encoded as $z$. This gives the paths from $(0,s)$ to their first
  zero (i.\,e., where it touches the horizontal axis for the first time).
\item We mark this zero by the symbol~$u$.
\item We either do a single ascent or a single decent (marked by a~$z$), then
  continue with a $\f{C}{z}$-block and do a single decent or ascent
  respectively (marked by a~$z$ as well) again. Thus, we are back at a zero.
\item We repeat such up/down blocks $z\f{C}{z}z$, each one preceded by a
  zero~$u$, a finite number of times.
\end{itemize}
If $s<0$, then the construction is the same, but everything is reflected at the
horizontal axis.

Continuing using the symbolic method---the description above is already part of
it, see, for example, Flajolet and
Sedgewick~\cite{Flajolet-Sedgewick:ta:analy}---the decomposition above
translates to the generating function
\begin{equation}
  \label{eq:path-gf-zeros}
  Q_s(z,u) = \frac{\f{C}{z}^{\abs{s}} z^{\abs{s}}}{1 - 2 u z^2 \f{C}{z}},
\end{equation}
which we will use from now on. Note that the coefficient $2$ reflects the fact
that there are two choices (up and down) for the blocks between zeros.

To obtain a nice explicit form, we perform a change of variables. The result is
stated in the following lemma.

\begin{lemma}\label{lem:transform-gf-zeros}
  With the transformation $z = v / (1+v^2)$ we have
  \begin{equation*}
    Q_s(z,u) = \frac{v^{\abs{s}}(1+v^2)}{1-v^2(2u-1)}.
  \end{equation*}
\end{lemma}

\begin{proof}
  Transforming the counting generating function of Catalan paths yields
  \begin{equation*}
    \f{C}{z} = 1+v^2.
  \end{equation*}
  Thus~\eqref{eq:path-gf-zeros} becomes
  \begin{equation*}
    Q_s(z,u) = 
    (1+v^2)^{\abs{s}}
    \Big(\frac{v}{1+v^2}\Big)^{\abs{s}}
    \frac{1}{1-2u\big(\frac{v}{1+v^2}\big)^2(1+v^2)}
  \end{equation*}
  and can be simplified to the expression stated in the lemma.
\end{proof}

The next step is to extract the coefficients out of the expressions obtained in
the previous lemma. First we rewrite the extraction of the coefficients from
the ``$z$-world'' to the ``$v$-world'', see
Lemma~\ref{lem:extract-coeffs-worlds}. Afterwards, in
Lemma~\ref{lem:coeffs-zeros}, the coefficients can be determined quite easily.

\begin{lemma}\label{lem:extract-coeffs-worlds}
  Let $F(z)$ be an analytic function in a neighborhood of the origin. Then we have
  \begin{equation*}
    [z^n] F(z) = [v^n] (1-v^2) (1+v^2)^{n-1}
      \f{F}{\frac{v}{1+v^2}}.
  \end{equation*}
\end{lemma}

\begin{proof}
  We use Cauchy's formula to extract the coefficients of $F(z)$ as
  \begin{equation*}
    [z^n] F(z) = \frac{1}{2\pi i}\oint_\calD \frac{dz}{z^{n+1}} F(z)
  \end{equation*}
  where $\calD$ is a positively oriented small circle around the origin. Under
  the transformation $z=v/(1+v^2)$, the circle $\calD$ is transformed to a
  contour $\calD'$ which still winds exactly once around the origin. Using
  Cauchy's formula again, we obtain
  \begin{align*}
    [z^n] F(z)
    &= \frac{1}{2\pi i}\oint_{\calD'} \frac{dv(1-v^2)}{(1+v^2)^2}
    \frac{(1+v^2)^{n+1}}{v^{n+1}} \f{F}{\frac{v}{1+v^2}} \\
    &= [v^n] (1-v^2) (1+v^2)^{n-1} \f{F}{\frac{v}{1+v^2}}.
  \end{align*}
\end{proof}

Now we are ready to calculate the desired coefficients.

\begin{lemma}\label{lem:coeffs-zeros}
  Suppose $n\equiv s\pmod 2$. Then we have
  \begin{equation*}
    [z^n] Q_s(z,1)
    = \binom{n}{(n-s)/2}
  \end{equation*}
  and, moreover,
  \begin{equation*}
    [z^n] \left.\frac{\partial}{\partial u} Q_s(z,u)\right\vert_{u=1}
    = 2 \sum_{k=0}^{(n-\abs{s})/2-1}\binom{n}{k}.
  \end{equation*}
\end{lemma}

\begin{proof}
  As $n \equiv s \pmod 2$, the number $n-s$ is even, and so we can
  set $\ell = \frac12(n-s)$. Then $[z^n]Q_s(z, 1)$ is the number of paths from
  $(0, s)$ to $(n, 0)$. These paths have $\ell$ up steps and $n-\ell$ down
  steps; thus there are $\binom{n}{\ell}$ many such paths.

  For the second part of this lemma, we restrict ourselves to $s\geq0$ (otherwise use $-s$ and the symmetry in~$s$
  of the generating function~\eqref{eq:path-gf-zeros} instead). We start with the result of
  Lemma~\ref{lem:transform-gf-zeros}. Taking the first derivative and setting
  $u=1$ yields
  \begin{equation*}
    \left.\frac{\partial}{\partial u} Q_s(z,u)\right\vert_{u=1}
    = \frac{2v^{s+2}(1+v^2)}{(1-v^2)^2}.
  \end{equation*}
  Thus, by using Lemma~\ref{lem:extract-coeffs-worlds}, we get
  \begin{equation*}
    [z^n] \frac{2v^{s+2}(1+v^2)}{(1-v^2)^2}
    = 2\,[v^{n-s-2}]\frac{(1+v^2)^{n}}{1-v^2}.
  \end{equation*}
  We use $\ell$ as above and get
  \begin{equation*}
    [v^{n-s-2}]\frac{(1+v^2)^n}{1-v^2}
    = [v^{2\ell-2}]\frac{(1+v^2)^n}{1-v^2}
    = [v^{\ell-1}] \frac{(1+v)^n}{1-v}
    = \sum_{k=0}^{\ell-1}\binom{n}{k},
  \end{equation*}
  which was claimed to hold.
\end{proof}

We are now ready to prove the main theorem (Theorem~\ref{thm:paths-gf-zeros})
of this section, which provides an expression for the expected number of
zeros. This exact expression is written as a double sum.

\begin{proof}[of Theorem~\ref{thm:paths-gf-zeros}]
  By Lemma~\ref{lem:coeffs-zeros}, the average number of zeros
  (except the zero at the end point) of
  a path of length~$n$ which starts in $(0,s)$ is
  \begin{equation*}
    \mu_{n,s} = \frac{[z^n] \left.\frac{\partial}{\partial u}
        Q_s(z,u)\right\vert_{u=1}}{[z^n] Q_s(z,1)}
    = \frac{2}{\binom{n}{\ell}} \sum_{k=0}^{\ell-1}\binom{n}{k},
  \end{equation*}
  where we have set $\ell = \frac12 (n-\abs{s})$ as in the proof of
  Lemma~\ref{lem:coeffs-zeros}. If $s=0$, this simplifies to
  \begin{equation}\label{eq:mu_n_0}
    \mu_{n,0} = \frac{2}{\binom{n}{n/2}} \sum_{k=0}^{n/2-1}\binom{n}{k}
    = \frac{2^n}{\binom{n}{n/2}} - 1.
  \end{equation}
  If $n\not\equiv s\pmod 2$, then we set $\mu_{n,s} = 0$.

  Summing up yields
  \begin{align*}
    \sum_{s=-n}^n \mu_{n,s}
    &= 2 \sum_{s=1}^n \mu_{n,s} + \mu_{n,0}
    = 4 \sum_{\ell=0}^{\ceil{n/2}-1}
    \frac{1}{\binom{n}{\ell}} \sum_{k=0}^{\ell-1}\binom{n}{k}
    + \mu_{n,0} \\
    &= 4 \sum_{0\leq k < \ell < \ceil{n/2}} \frac{\binom{n}{k}}{\binom{n}{\ell}}
    + \iverson*{$n$ even} 
    \left(\frac{2^n}{\binom{n}{n/2}} - 1\right).
  \end{align*}
  Dividing by the number $n+1$ of possible starting points and adding
  $1$ for the zero at $(n,0)$ completes the proof of
  Theorem~\ref{thm:paths-gf-zeros}.
\end{proof}


\section{Appendix to Section~\ref{sec:prob}: A Probabilistic Approach}
\label{sec:appendix:prob}

The following remark explains the equivalence between the urn model
(Section~\ref{sec:prob}) and the lattice path model
(Section~\ref{sec:description}).

\begin{remark}\label{rem:equivalence-urn-paths}
\appendixequrnpaths
\end{remark}

\appendixproofprobuniform


\appendixidentity


\appendixasy


\appendixdistribution


\section{Appendix to Section~\ref{sec:more-zeros}: Going to Zero and Coming From Zero}
\label{sec:appendix:more-zeros}

\appendixtofromzerodef
\appendixtofromzero


\section{Appendix to Section~\ref{sec:lattice-paths-N}: Lattice Paths of Variable Length}
\label{sec:appendix:lattice-paths-N}

\appendixvarlen


\section{Appendix to Section~\ref{sec:solve-recurrence}}
\label{sec:appendix:quicksort:recurrence}

\appendixqsrecurrence


\section{Appendix to Section~\ref{sec:solve-recurrence}: Solving the Dual-Pivot Quicksort Recurrence}
\label{sec:appendix:solve-recurrence}

\appendixrecint


\section{Appendix to Section~\ref{sec:part-costs}: Partitioning Algorithms and Their Cost}
\label{sec:appendix:part-costs}

\appendixpartcost
\appendixpcost


\section{Appendix to Section~\ref{sec:costs-main-asy}: Main Results and their Asymptotic Aspects}
\label{sec:appendix:costs-main-asy}

\appendixproofmain


\appendixexperiments

\algrenewcommand{\algorithmiccomment}[1]{\hskip3em // #1}

\section{Pseudocode of Dual-Pivot Quicksort Algorithms}
\label{sec:pseudocode}
In this supplementary section, we give the full pseudocode for the 
strategies ``Clairvoyant'' (Algorithm~\ref{algo:clairvoyant}) and ``Count'' (Algorithm~\ref{algo:count}) turned into dual-pivot quicksort algorithms.
\renewcommand{\alglinenumber}[1]{\footnotesize{#1} }

\begin{algorithm}
    \caption{Dual-Pivot Quicksort Algorithm ``Clairvoyant''}\samepage\label{algo:clairvoyant}
    \textbf{procedure} \textit{Clairvoyant}($\textit{A}$, $\textit{left}$, $\textit{right}$)
    \medskip
    \begin{algorithmic}[1]
        \If{$\textit{right} \leq \textit{left}$}
            \State \Return
        \EndIf
        \If{$A[\textit{right}] < A[\textit{left}]$}
        \State swap \textit{A}[\textit{left}] and \textit{A}[\textit{right}]
        \EndIf
        \State $\texttt{p} \gets A[\textit{left}]$
        \State $\texttt{q} \gets A[\textit{right}]$
        \State $\texttt{i} \gets \textit{left} + 1$;
               $\texttt{k} \gets \textit{right} - 1$;
               $\texttt{j} \gets \texttt{i}$
        \State $\texttt{d} \gets \#(\text{small elements}) - \#(\text{large elements})$
        \Comment{$\texttt{d}$ is given by an oracle.}
        \While{$\texttt{j} \leq \texttt{k}$}
            \If{$\texttt{d} \geq 0$}
                \If{$\textit{A}[\texttt{j}] < \textit{p}$}
                    \State swap $\textit{A}[\texttt{i}]$ and $\textit{A}[\texttt{j}]$
                    \State $\texttt{i} \gets \texttt{i} + 1$;
                           $\texttt{j} \gets \texttt{j} + 1$;
                           $\texttt{d} \gets \texttt{d} - 1$
                \Else
                    \If{$\textit{A}[\texttt{j}] < \textit{q}$}
                        \State $\texttt{j} \gets \texttt{j} + 1$
                    \Else
                        \State swap $\textit{A}[\texttt{j}]$ and $\textit{A}[\texttt{k}]$
                        \State $\texttt{k} \gets \texttt{k} - 1$;
                               $\texttt{d} \gets \texttt{d} + 1$
                    \EndIf
                \EndIf
            \Else
                \If{$\textit{A}[\texttt{k}] > \textit{q}$}
                    \State $\texttt{k} \gets \texttt{k} - 1$;
                           $\texttt{d} \gets \texttt{d} + 1$
                \Else
                    \If{$\textit{A}[\texttt{k}] < \textit{p}$}\\
                                            \hskip3em\Comment{Perform a cyclic rotation to the left, i.\,e.,}\\
                        \hskip3em\Comment{$\texttt{tmp} \gets \textit{A}[\texttt{k}]$;
                                          $\textit{A}[\texttt{k}] \gets \textit{A}[\texttt{j}]$;
                                          $\textit{A}[\texttt{j}] \gets \textit{A}[\texttt{i}]$;
                                          $\textit{A}[\texttt{i}] \gets \texttt{tmp}$}
                        \State \textit{rotate3}($\textit{A}[\texttt{k}], \textit{A}[\texttt{j}],
                                        \textit{A}[\texttt{i}]$)
                        \State $\texttt{i} \gets \texttt{i} + 1$;
                               $\texttt{d} \gets \texttt{d} - 1$
                    \Else
                        \State swap $\textit{A}[\texttt{j}]$ and $\textit{A}[\texttt{k}]$
                    \EndIf
                    \State $\texttt{j} \gets \texttt{j} + 1$
                \EndIf
            \EndIf
        \EndWhile
        \State swap $\textit{A}[\textit{left}]$ and $\textit{A}[\texttt{i}-1]$
        \State swap $\textit{A}[\textit{right}]$ and $\textit{A}[\texttt{k}+1]$
        \State \textit{Clairvoyant}(\textit{A}, \textit{left}, $\texttt{i} - 2$)
        \State \textit{Clairvoyant}(\textit{A}, $\texttt{i}$, $\texttt{k}$)
        \State \textit{Clairvoyant}(\textit{A}, $\texttt{k}$ + 2, \textit{right})
    \end{algorithmic}
\end{algorithm}

\begin{algorithm}
    \caption{Dual-Pivot Quicksort Algorithm ``Count''}\samepage\label{algo:count}
    \textbf{procedure} \textit{Count}($\textit{A}$, $\textit{left}$, $\textit{right}$)
    \medskip
    \begin{algorithmic}[1]
        \If{$\textit{right} \leq \textit{left}$}
            \State \Return
        \EndIf
        \If{$A[\textit{right}] < A[\textit{left}]$}
        \State swap \textit{A}[\textit{left}] and \textit{A}[\textit{right}]
        \EndIf
        \State $\texttt{p} \gets A[\textit{left}]$
        \State $\texttt{q} \gets A[\textit{right}]$
        \State $\texttt{i} \gets \textit{left} + 1$;
               $\texttt{k} \gets \textit{right} - 1$;
               $\texttt{j} \gets \texttt{i}$
        \State $\texttt{d} \gets 0$
        \Comment{$\texttt{d}$ holds the difference of the number of small and large elements.}
        \While{$\texttt{j} \leq \texttt{k}$}
            \If{$\texttt{d} \geq 0$}
                \If{$\textit{A}[\texttt{j}] < \textit{p}$}
                    \State swap $\textit{A}[\texttt{i}]$ and $\textit{A}[\texttt{j}]$
                    \State $\texttt{i} \gets \texttt{i} + 1$;
                           $\texttt{j} \gets \texttt{j} + 1$;
                           $\texttt{d} \gets \texttt{d} + 1$
                \Else
                    \If{$\textit{A}[\texttt{j}] < \textit{q}$}
                        \State $\texttt{j} \gets \texttt{j} + 1$
                    \Else
                        \State swap $\textit{A}[\texttt{j}]$ and $\textit{A}[\texttt{k}]$
                        \State $\texttt{k} \gets \texttt{k} - 1$;
                               $\texttt{d} \gets \texttt{d} - 1$
                    \EndIf
                \EndIf
            \Else
                \If{$\textit{A}[\texttt{k}] > \textit{q}$}
                    \State $\texttt{k} \gets \texttt{k} - 1$;
                           $\texttt{d} \gets \texttt{d} - 1$
                \Else
                    \If{$\textit{A}[\texttt{k}] < \textit{p}$}
                        \State \textit{rotate3}($\textit{A}[\texttt{k}], \textit{A}[\texttt{j}],
                                        \textit{A}[\texttt{i}]$)
                        \State $\texttt{i} \gets \texttt{i} + 1$;
                               $\texttt{d} \gets \texttt{d} + 1$
                    \Else
                        \State swap $\textit{A}[\texttt{j}]$ and $\textit{A}[\texttt{k}]$
                    \EndIf
                    \State $\texttt{j} \gets \texttt{j} + 1$
                \EndIf
            \EndIf
        \EndWhile
        \State swap $\textit{A}[\textit{left}]$ and $\textit{A}[\texttt{i}-1]$
        \State swap $\textit{A}[\textit{right}]$ and $\textit{A}[\texttt{k}+1]$
        \State \textit{Count}(\textit{A}, \textit{left}, $\texttt{i} - 2$)
        \State \textit{Count}(\textit{A}, $\texttt{i}$, $\texttt{k}$)
        \State \textit{Count}(\textit{A}, $\texttt{k}$ + 2, \textit{right})
    \end{algorithmic}
\end{algorithm}
 

\clearpage

\end{document}